\documentclass[11pt]{article}
\usepackage{amssymb,amsmath}
\usepackage{mathrsfs}
\usepackage{latexsym}
\usepackage{color}
\usepackage{amssymb}
\usepackage{amsmath}
\usepackage{amsthm}

\allowdisplaybreaks

\catcode`!=11
\let\!int\int \def\int{\displaystyle\!int}
\let\!lim\lim \def\lim{\displaystyle\!lim}
\let\!sum\sum \def\sum{\displaystyle\!sum}
\let\!sup\sup \def\sup{\displaystyle\!sup}
\let\!inf\inf \def\inf{\displaystyle\!inf}
\let\!cap\cap \def\cap{\displaystyle\!cap}
\let\!max\max \def\max{\displaystyle\!max}
\let\!min\min \def\min{\displaystyle\!min}
\catcode`!=12

\newtheorem{theorem}{\bf Theorem}[section]
\newtheorem{lemma}{\bf Lemma}[section]
\newtheorem{definition}{\bf Definition}[section]
\newtheorem{proposition}{\bf Proposition}[section]
\newtheorem{corollary}{\bf Corollary}[section]
\newtheorem{remark}{\bf Remark}[section]

\catcode`@=11
\@addtoreset{equation}{section}
\catcode`@=12

\begin{document}
\title{Stochastic averaging principle for multiscale stochastic linearly coupled complex
cubic-quintic Ginzburg-Landau equations\footnote{The first author is supported by NSFC Grant 11601073 and the Fundamental Research Funds
for the Central Universities, the second author is supported by NSFC Grant 11171132.}}
\author{Peng Gao, Yong Li
\\[2mm]
\small School of Mathematics and Statistics, and Center for Mathematics
\\
\small and Interdisciplinary Sciences, Northeast Normal University,
\\
\small Changchun 130024,  P. R. China
\\[2mm]
\small Email: gaopengjilindaxue@126.com }
\date{}
\maketitle

\vbox to -13truemm{}

\begin{abstract}
Stochastic averaging principle is a powerful tool for studying qualitative analysis of stochastic dynamical systems
with different time-scales. In this paper, we will establish an averaging principle for multiscale stochastic linearly coupled complex
cubic-quintic Ginzburg-Landau equations
with slow and fast time scales. Under suitable conditions, the existence of an averaging equation eliminating the fast variable for this coupled system is proved, and as a consequence, the system can be reduced to a single stochastic complex
cubic-quintic Ginzburg-Landau equation with a modified coefficient.
\\[6pt]
\end{abstract}
\section{Introduction}
\par
Linearly coupled complex
cubic-quintic Ginzburg-Landau equations
\begin{equation*}
\begin{array}{l}
\left\{
\begin{array}{llll}
A_{t}+(-i-\beta)\triangle A+(-i-\gamma)|A|^{2}A+(-\mu-i\nu)|A|^{4}A-\eta A-i\kappa B=0
\\B_{t}+(-i-\beta)\triangle B+(-i-\gamma)|B|^{2}B+(-\mu-i\nu)|B|^{4}B-\eta B-i\kappa A=0
\end{array}
\right.
\end{array}
\end{equation*}
supply a model for ring
lasers based on dual-core optical fibers \cite{A1,M1,M2, W2,W3}, and are noteworthy
dynamical systems by themselves \cite{S1}, the model is also of direct relevance
to optics, as it describes a ring laser based on a dual-core fiber with bandwidth-limited gain in both cores (which is modeled
by the combination of the linear and quintic losses and cubic gain).
\par
It is said as in \cite{S2} that almost all physical systems
have a certain hierarchy in which not all components evolve at the same rate, i.e., some of components vary very rapidly, while others change very slowly.
In the linearly coupled complex cubic-quintic Ginzburg-Landau equations, the components $A,B$ may vary very differently,
namely, one is slow component and the other one is fast component.

\par
The averaging principle is an important method to extract effective macroscopic dynamic from complex systems with slow component and fast component.
In this paper, we will be concerned with the stochastic averaging principle for multiscale stochastic linearly coupled complex
cubic-quintic Ginzburg-Landau equations
\begin{equation*}
\begin{array}{l}
(\ast)\left\{
\begin{array}{llll}
dA^{\varepsilon}+[(-i-\beta)A^{\varepsilon}_{xx}+(-i-\gamma)|A^{\varepsilon}|^{2}A^{\varepsilon}+(-\mu-i\nu)|A^{\varepsilon}|^{4}A^{\varepsilon}-\eta A^{\varepsilon}-i\kappa B^{\varepsilon}]dt=\sigma_{1}dW_{1}
\\dB^{\varepsilon}+\frac{1}{\varepsilon}[(-i-\beta)B^{\varepsilon}_{xx}+(-i-\gamma)|B^{\varepsilon}|^{2}B^{\varepsilon}+(-\mu-i\nu)|B^{\varepsilon}|^{4}B^{\varepsilon}-\eta B^{\varepsilon}-i\kappa A^{\varepsilon}]dt=\frac{1}{\sqrt{\varepsilon}}\sigma_{2}dW_{2}
\\A^{\varepsilon}(0,t)=0=A^{\varepsilon}(1,t)
\\B^{\varepsilon}(0,t)=0=B^{\varepsilon}(1,t)
\\A^{\varepsilon}(x,0)=A_{0}(x)
\\B^{\varepsilon}(x,0)=B_{0}(x)
\end{array}
\right.
\end{array}
\begin{array}{lll}
{\rm{in}}~Q\\
{\rm{in}}~Q\\
{\rm{in}}~(0,T)\\
{\rm{in}}~(0,T)\\
{\rm{in}}~I\\
{\rm{in}}~I,
\end{array}
\end{equation*}
where $T>0, I=(0,1), Q=I\times (0,T),$ the stochastic perturbations are of additive type, $W_{1}$ and $W_{2}$ are
mutually independent Wiener processes on a complete stochastic basis $(\Omega,\mathcal{F},\mathcal{F}_{t},\mathbb{P})$, which
will be specified later, we denote by $\mathbb{E}$ the expectation with respect to $\mathbb{P}$. The noise coefficients $\sigma_{1}$ and $\sigma_{2}$ are positive constants and the parameter
$\varepsilon$ is small and positive, which describes the ratio of time scale between the process $A^{\varepsilon}$ and
$B^{\varepsilon}$. With this time scale the variable $A^{\varepsilon}$ is referred as slow component and $B^{\varepsilon}$ as the
fast component.

\par
Many problems in the natural sciences give rise to singularly
perturbed systems of stochastic partial differential equations. In the past four
decades, singularly perturbed systems have been the focus of extensive research within
the framework of averaging methods. The separation of scales is then taken to advantage
to derive a reduced equation, which approximates the slow components. Conditions
under which the averaging principle can be applied to this kind of system are
well known in the classical literature.
\par
Multiscale stochastic partial differential equations(SPDEs) arise as models for various complex
systems, such model arises from describing multiscale phenomena in, for example, nonlinear oscillations,
material sciences, automatic control, fluids dynamics, chemical kinetics and in other areas leading to mathematical description
involving ``slow'' and ``fast'' phase variables. The study of the asymptotic
behavior of such systems is of great interest. In this respect, the question of how the physical effects at large time scales
influence the dynamics of the system is arisen. We focus on this question and show
that, under some dissipative conditions on fast variable equation, the complexities effects at
large time scales to the asymptotic behavior of the slow component can be omitted or neglected
in some sense.
\par
The theory of stochastic averaging principle provides an effective approach for the
qualitative analysis of stochastic systems with different time-scales and is relatively
mature for stochastic dynamical systems.
The theory of averaging principle serves as a tool in study of the qualitative behaviors for complex systems with multiscales, it is essential for describing and understanding the asymptotic behavior of dynamical systems with fast and slow variables. Its basic idea is to approximate the original
system by a reduced system. The theory of averaging for deterministic dynamical systems, which
was first studied by Bogoliubov \cite{B0}, has a long and rich history.

\par
The averaging principle in the stochastic
differential equations(SDEs) setup was first considered by Khasminskii \cite{K1} which proved that
an averaging principle holds in weak sense, and has been an active research field on which there is a great deal of literature.
Taking into account the generalized and refined results, it is worthy quoting the paper
\par
$\bullet$ convergence in probability: Veretennikov \cite{V1,V2}, Freidlin and Wentzell \cite{F5,F6}; .
\par
$\bullet$ mean-square type convergence: Golec and Ladde \cite{G1}, Givon and co-workers \cite{G2};
\par
$\bullet$ strong convergence: Givon \cite{G3} and Golec \cite{G4}.
\par
We are also referred to \cite{L2} and the references therein for recent related work on averaging for stochastic systems in finite
dimensional space.
\par
However, there are few results on the averaging principle for stochastic systems in infinite dimensional space, an important contribution in this direction has been given by Cerrai and Freidlin with their paper \cite{C1} which appeared in 2009. To the best of our knowledge, this is the first article on the averaging principle for stochastic systems in infinite dimensional space, it presented an averaged principle for slow-fast stochastic reaction-diffusion equations.
\par
Next, we recall the recent results:
\par
$\bullet$ convergence in weak sense (convergence in law): Br\'{e}hier \cite{B2}, Dong and co-workers \cite{D2}, Fu and co-workers \cite{F7};
\par
$\bullet$ convergence in probability:  Cerrai and Freidlin \cite{C1}, Cerrai \cite{C2,C3};
\par
$\bullet$ strong convergence: Wang and Roberts \cite{W1},  Br\'{e}hier \cite{B2}, Fu and co-workers \cite{F1,F2,F3,F4}, Dong and co-workers \cite{D2},
Xu and co-workers \cite{X1,X2}, Bao and co-workers \cite{B4}, Pei and co-workers \cite{P3} .
\par
Almost all the above papers considered the stochastic reaction-diffusion
equations or stochastic hyperbolic-parabolic equations, the nonlinear terms are assumed to be Lipschitz continuous and in
particular to have linear growth. \cite{D2} considers the averaging principle for one dimensional stochastic Burgers
equation, this is the first article to deal with highly nonlinear
term on this topic.
\par
However, to the best of our knowledge, the averaging principle for the stochastic linearly coupled complex
cubic-quintic Ginzburg-Landau equations $(\ast)$ has not been so far solved, a natural question is as follows:
\par
~~
\par $\bigstar$ \textit{Can we establish the averaging principle for the stochastic linearly coupled complex
cubic-quintic Ginzburg-Landau equations $(\ast)$ ? To be more precise, can the slow component $A^{\varepsilon}$  be approximated by the solution $\bar{A}$
which governed by a stochastic cubic-quintic Ginzburg-Landau equation?
}
\par
~~
\par
The main object in this paper is to establish an
effective approximation for slow process $A^{\varepsilon}$ with respect to the limit $\varepsilon\rightarrow0$. The main difficulity in this paper is the cubic nonlinear terms and the quintic nonlinear terms.
\par
In this paper, we will take
\begin{equation*}
\begin{array}{l}
\begin{array}{llll}
\mu=\gamma=-1,\nu=1
\end{array}
\end{array}
\end{equation*}
for the sake of simplicity. All the results
can be extended without difficulty to the general case.
\par
We define
\begin{equation*}
\begin{array}{l}
\begin{array}{llll}
\mathcal{L}(A)=(i+\beta)A_{xx},
\\
\mathcal{F}(A)=(i+\gamma)|A|^{2}A=(-1+i)|A|^{2}A,
\\
\mathcal{G}(A)=(i\nu+\mu)|A|^{4}A=(-1+i)|A|^{4}A,
\\
f(A,B)=\eta A +i\kappa B,
\\
g(A,B)=\eta B +i\kappa A,
\end{array}
\end{array}
\end{equation*}
then the linearly coupled complex
cubic-quintic Ginzburg-Landau equations $(\ast)$ becomes
\begin{equation}\label{1}
\begin{array}{l}
\left\{
\begin{array}{llll}
dA^{\varepsilon}=[\mathcal{L}(A^{\varepsilon})+\mathcal{F}(A^{\varepsilon})+\mathcal{G}(A^{\varepsilon})+f(A^{\varepsilon}, B^{\varepsilon})]dt+\sigma_{1}dW_{1}
\\dB^{\varepsilon}=\frac{1}{\varepsilon}[\mathcal{L}(B^{\varepsilon})+\mathcal{F}(B^{\varepsilon})+\mathcal{G}(B^{\varepsilon})+g(A^{\varepsilon}, B^{\varepsilon})]dt+\frac{1}{\sqrt{\varepsilon}}\sigma_{2}dW_{2}
\\A^{\varepsilon}(0,t)=0=A^{\varepsilon}(1,t)
\\B^{\varepsilon}(0,t)=0=B^{\varepsilon}(1,t)
\\A^{\varepsilon}(x,0)=A_{0}(x)
\\B^{\varepsilon}(x,0)=B_{0}(x)
\end{array}
\right.
\end{array}
\begin{array}{lll}
{\rm{in}}~Q,\\
{\rm{in}}~Q,\\
{\rm{in}}~(0,T),\\
{\rm{in}}~(0,T),\\
{\rm{in}}~I,\\
{\rm{in}}~I.
\end{array}
\end{equation}
\subsection{Mathematical setting}
\par We introduce the following mathematical setting:
\par
$\diamond$ We denote by $L^{2}(I)$ the space of all Lebesgue square integrable
complex-valued functions on $I$. The inner product on $L^{2}(I)$ is
\begin{eqnarray*}
( u,v)=\Re\int_{I}u\overline{v}dx,
\end{eqnarray*}
for any $u,v\in L^{2}(I),$ where $\overline{\bullet}$ denotes the conjugate of $\bullet.$ The norm on $L^{2}(I)$ is
\begin{eqnarray*}
\|u\|=( u,u )^{\frac{1}{2}},
\end{eqnarray*}
for any $u\in L^{2}(I).$
\par
$H^{s}(I)(s\geq 0)$ are the classical Sobolev
spaces of complex-valued functions on $I$. The definition of $H^{s}(I)$ can be found in \cite{L1}, the norm on $H^{s}(I)$ is
$\|\cdot\|_{H^{s}}.$
\par
We set
\begin{eqnarray*}
\begin{array}{l}
\begin{array}{llll}
X_{p,\tau}=L^{p}(\Omega;C([0,\tau];H^{1}(I))\times L^{p}(\Omega;C([0,\tau];H^{1}(I)),\\
Y_{\tau}=C([0,\tau];H^{1}(I)\times C([0,\tau];H^{1}(I),
\end{array}
\end{array}
\end{eqnarray*}
where $p\geq1,\tau\geq0.$ The norms on $X_{p,\tau}$ and $Y_{\tau}$ are defined as
\begin{eqnarray*}
\begin{array}{l}
\begin{array}{llll}
\|(u,v)\|_{X_{p,\tau}}=\|u\|_{L^{p}(\Omega;C([0,\tau];H^{1}(I))}+\|v\|_{L^{p}(\Omega;C([0,\tau];H^{1}(I))},\\
\|(u,v)\|_{Y_{\tau}}=\|u\|_{C([0,\tau];H^{1}(I))}+ \|v\|_{C([0,\tau];H^{1}(I))}.
\end{array}
\end{array}
\end{eqnarray*}
\par
$\diamond$ For $i=1,2,$ let $\{e_{i,k}\}_{k\in \mathbb{N}}$ be eigenvectors of a nonnegative, symmetric operator $Q_{i}$ with corresponding eigenvalues $\{\lambda_{i,k}\}_{k\in \mathbb{N}}$, such that
\begin{eqnarray*}
\begin{array}{l}
\begin{array}{llll}
Q_{i}e_{i,k}=\lambda_{i,k}e_{i,k},~~\lambda_{i,k}>0,~k\in \mathbb{N}.
\end{array}
\end{array}
\end{eqnarray*}
Let $W_{i}$ be an $L^{2}(I)-$valued $Q_{i}$-Wiener process with operator $Q_{i}$ satisfying
\begin{eqnarray*}
\begin{array}{l}
\begin{array}{llll}
TrQ_{i}=\sum\limits_{k=1}^{+\infty}\lambda_{i,k}<+\infty,~~k\in \mathbb{N}
\end{array}
\end{array}
\end{eqnarray*}
and
\begin{eqnarray*}
\begin{array}{l}
\begin{array}{llll}
W_{i}=\sum\limits_{k=1}^{+\infty}\lambda_{i,k}^{\frac{1}{2}}\beta_{i,k}(t)e_{i,k}<+\infty,~~k\in \mathbb{N}~~t\geq0,
\end{array}
\end{array}
\end{eqnarray*}
where $\{\beta_{i,k}\}_{k\in \mathbb{N}}(i=1,2)$ are independent real-valued Brownian motions on the probability base $(\Omega,\mathcal{F},\mathcal{F}_{t},\mathbb{P})$.
\par
We define $\|a\|_{Q_{i}}^{2}\triangleq a^{2}Tr Q_{i},$ for $i=1,2.$
\par
$\diamond$ Throughout the paper, the letter $C$ denotes positive constants
whose value may change in different occasions. We will write the dependence
of constant on parameters explicitly if it is essential.
\par We adopt the following hypothesis (H) throughout this paper:
\par
(H) $\beta>0,\eta>0,\kappa\in \mathbb{R},\alpha\triangleq\frac{\beta\lambda}{2}-\eta>0,$ where $\lambda>0$ is the smallest constant such that the following inequality holds
\begin{eqnarray*}
\begin{array}{l}
\begin{array}{llll}
\|u_{x}\|^{2}\geq\lambda\|u\|^{2},
\end{array}
\end{array}
\end{eqnarray*}
where $u\in H_{0}^{1}(I)$ or $\int_{I}udx=0.$
\subsection{Main results}
\subsubsection{Well-posedness for (\ref{1})}
\par
Let us explain what we mean by a solution of the multiscale stochastic linearly
coupled complex cubic-quintic Ginzburg-Landau equations in this article.
\begin{definition}
If $(A^{\varepsilon},B^{\varepsilon})$ is an adapted process over $(\Omega,\mathcal {F},\{\mathcal {F}_{t}\}_{t\geq0},\mathbb{P})$ such that $\mathbb{P}-$a.s. the
integral equations
\begin{eqnarray*}
\begin{array}{l}
\begin{array}{llll}

A^{\varepsilon}(t)=S(t)A_{0}+\int_{0}^{t}S(t-s)(\mathcal{F}(A^{\varepsilon})+\mathcal{G}(A^{\varepsilon})+f(A^{\varepsilon}, B^{\varepsilon})) (s) ds+\int_{0}^{t}S(t-s)\sigma_{1}dW_{1}
\\B^{\varepsilon}(t)=S(\frac{t}{\varepsilon})B_{0}+\frac{1}{\varepsilon}\int_{0}^{t}S(\frac{t-s}{\varepsilon})(\mathcal{F}(B^{\varepsilon})+\mathcal{G}(B^{\varepsilon})+g(A^{\varepsilon}, B^{\varepsilon}))(s)ds+\frac{1}{\sqrt{\varepsilon}}\int_{0}^{t}S(\frac{t-s}{\varepsilon})\sigma_{2}dW_{2}
\end{array}
\end{array}
\end{eqnarray*}
hold true for all $t\in [0,T],$ we say that it is a mild solution for (\ref{1}).
\end{definition}
\par
Now, we are in a position to present the first main result in this paper.
\begin{theorem}\label{Th2}
Suppose that the hypothesis (H) holds, for any $\varepsilon\in(0,1),T>0,$ if $(A_{0},B_{0})\in H_{0}^{1}(I)\times H_{0}^{1}(I),$ (\ref{1}) admits a unique mild solution $(A^{\varepsilon},B^{\varepsilon})\in X_{2,T}.$
\end{theorem}
\par
The Banach
contraction principle is used as the main tool for proving the existence of mild solutions of SPDE
in most of the existing papers. We first apply the fixed point theorem to the corresponding
truncated equation and give the local existence of mild solution to (\ref{1}). Then, the energy
estimates show that the solution is also global in time.
\subsubsection{Stochastic averaging principle for (\ref{1})}
\par
Asymptotical methods play an important role in investigating
nonlinear dynamical systems. In particular, the averaging
methods provide a powerful tool for simplifying dynamical systems,
and obtain approximate solutions to differential equations
arising from mechanics, mathematics, physics, control and
other areas. In this paper, we use stochastic averaging principle to investigate
stochastic linearly coupled complex cubic-quintic Ginzburg-Landau equations (\ref{1}).
\par
Now, we are in a position to present the second main result in this paper.
\begin{theorem}\label{Th1}
Suppose that the hypothesis (H) holds and $A_{0},B_{0}\in H^{1}_{0}(I),$ $(A^{\varepsilon},B^{\varepsilon})$ is the solution of (\ref{1}) and $\bar{A}$ is the solution of the effective dynamics equation
\begin{equation}\label{8}
\begin{array}{l}
\left\{
\begin{array}{llll}
d\bar{A}=[\mathcal{L}(\bar{A})+\mathcal{F}(\bar{A})+\mathcal{G}(\bar{A})+\bar{f}(\bar{A})]dt+\sigma_{1}dW_{1}
\\\bar{A}(0,t)=0=\bar{A}(1,t)
\\\bar{A}(x,0)=A_{0}(x)
\end{array}
\right.
\end{array}
\begin{array}{lll}
{\rm{in}}~Q\\
{\rm{in}}~(0,T)\\
{\rm{in}}~I,
\end{array}
\end{equation}
then for any $T>0,$ any $p>0,$ we have
\begin{eqnarray*}
\begin{array}{l}
\begin{array}{llll}
\lim\limits_{\varepsilon\rightarrow 0}\mathbb{E}\sup\limits_{t\in[0,T]}\|A^{\varepsilon}(t)-\bar{A}(t)\|^{2p}=0,
\end{array}
\end{array}
\end{eqnarray*}
where
\begin{eqnarray*}
\begin{array}{l}
\begin{array}{llll}
\bar{f}(A)=\int_{L^{2}(I)}f(A,B)\mu^{A}(dB)
\end{array}
\end{array}
\end{eqnarray*}
and $\mu^{A}$ is an invariant measure for the fast motion with frozen slow component
\begin{equation}\label{3}
\begin{array}{l}
\left\{
\begin{array}{llll}
dB=[\mathcal{L}(B)+\mathcal{F}(B)+\mathcal{G}(B)+g(A,B)]dt+\sigma_{2}dW_{2}
\\B(0,t)=0=B(1,t)
\\B(x,0)=B_{0}(x)
\end{array}
\right.
\end{array}
\begin{array}{lll}
{\rm{in}}~Q\\
{\rm{in}}~(0,T)\\
{\rm{in}}~I,
\end{array}
\end{equation}
where $A\in L^{2}(I).$
\par
Moreover, if $p>\frac{5}{4},$ there exists a positive constant $C(p)$ such that
\begin{equation*}
\begin{array}{l}
\begin{array}{llll}
\mathbb{E}(\sup\limits_{0\leq t\leq T}\|A^{\varepsilon}(t)-\bar{A}(t)\| ^{2p})
\leq C(p)(\frac{1}{-\ln \varepsilon})^{\frac{1}{8p}};
\end{array}
\end{array}
\end{equation*}
if $0<p\leq\frac{5}{4},$ for any $\kappa>0,$ there exists a positive constant $C(\kappa)$ such that
\begin{eqnarray*}
\begin{array}{l}
\begin{array}{llll}
\mathbb{E}(\sup\limits_{0\leq t\leq T}\|A^{\varepsilon}(t)-\bar{A}(t)\| ^{2p})
\leq C(\kappa)(\frac{1}{-\ln \varepsilon})^{\frac{2p}{(5+2\kappa)^{2}}}

.
\end{array}
\end{array}
\end{eqnarray*}

\end{theorem}
\begin{remark}
Our results show that the asymptotic behavior of (\ref{1}) can
be characterized by (\ref{3}) with averaged coefficients.
\end{remark}
\par
The main strategy for proving Theorem \ref{Th1} is:
\par We can establish stochastic averaging principle for (\ref{1}), this relies on the moment estimates and the Khasminskii technique already known for SDEs: we introduce an auxiliary process for which the slow component of the fast variable is frozen on small intervals of a subdivision. The introduction of the auxiliary process $(\hat{A}^{\varepsilon},\hat{B}^{\varepsilon})$ provides an intermediate step between the processes $A^{\varepsilon}$ and $\bar{A}$ whose difference we need to estimate.
\par
First, we establish the H\"{o}lder continuity of time variable for $A^{\varepsilon}$ which is a crucial step, this relies on the porperty of semigroup $\{S(t)\}_{t\geq 0}$ and the moment estimates of $(A^{\varepsilon},B^{\varepsilon})$.
\par
Second, based on this H\"{o}lder continuity property, the errors of $A^{\varepsilon}-\hat{A}^{\varepsilon}$ and $B^{\varepsilon}-\hat{B}^{\varepsilon}$ can be obtained, we will establish convergence of the auxiliary process $\hat{B}^{\varepsilon}$ to the fast solution process $B^{\varepsilon}$ and $\hat{A}^{\varepsilon}$ to the slow solution process $A^{\varepsilon}$, respectively.
\par
Third, by using the skill of stopping times which was introduced in \cite{D2} and the moment estimates of $A^{\varepsilon},B^{\varepsilon},\hat{A}^{\varepsilon},\hat{B}^{\varepsilon}$, we can establish the errors of $\hat{A}^{\varepsilon}-\bar{A}.$
\par
Finally, we can establish the errors of $A^{\varepsilon}-\bar{A},$ and we arrive at Theorem \ref{Th1}.

\subsection{Main novelties}
\par
The main novelties of this paper are twofold:
\par
$\bullet$
The first one is to extend the stochastic averaging principle result to stochastic linearly
coupled complex cubic-quintic Ginzburg-Landau equations (\ref{1}).
\par
The previous stochastic averaging principle were established for the nonlinear coupled heat-heat equations and the nonlinear coupled wave-heat equations which are
different from the one in this paper.
\par
$\bullet$
The second one is to overcome the no-Lipschitz property of the nonlinear term in (\ref{1}).
\par
The main difficulity in this paper is cubic nonlinear terms $|A|^{2}A,|B|^{2}B$ and quintic nonlinear terms $|A|^{4}A,|B|^{4}B$ in (\ref{1}) which are not Lipschitz-continuity, traditional methods can't deal with the difficulty in our problem, thus we need to take new measures. How to treat the cubic nonlinear terms and quintic nonlinear terms is the key of the paper.
\par
We overcome this difficulty by the semigroup approach, stochastic analysis techniques, the skill of stopping times, energy estimate method and refined inequality technique. The crucial tool is Proposition \ref{P1} which play a vital role in this article, namely the following moment estimates
\begin{eqnarray*}
\begin{array}{l}
\begin{array}{llll}
\sup\limits_{\varepsilon\in (0,1)}\sup\limits_{t\in[0,T]}\mathbb{E}\|A^{\varepsilon}(t)\|^{2p}
,
\sup\limits_{\varepsilon\in (0,1)}\sup\limits_{t\in[0,T]}\mathbb{E}\|B^{\varepsilon}(t)\|^{2p}
,
\\
\sup\limits_{\varepsilon\in (0,1)}\sup\limits_{t\in[0,T]}\mathbb{E}\|A^{\varepsilon}(t)\|_{H^{1}}^{2p}
,
\sup\limits_{\varepsilon\in (0,1)}\mathbb{E}\int_{0}^{T}\|B^{\varepsilon}(t)\|_{H^{1}}^{2p}dt
,
\\
\sup\limits_{\varepsilon\in (0,1)}\mathbb{E}\sup\limits_{0\leq t\leq T}\|A^{\varepsilon}(t)\|_{H^{1}}^{2p}.
\end{array}
\end{array}
\end{eqnarray*}
These moment estimates will be realized by stochastic tools under suitable assumptions, for example, It\^{o}'s formula, Burkholder-Davis-Gundy inequality, energy formula Young inequality and H\"{o}lder's inequality, etc.
\par
~~
\par
~~
\par

This paper is organized as follows.
In Sec. 2, we present
some preliminary results.
The fast motion equation (\ref{3}) is study in Sec. 3, we present an exponential ergodicity of a fast equation with the frozen slow
component. In Sec. 4, we establish the well-posedness and some a priori estimates for the slow-fast system (\ref{1}) and averaged equation (\ref{8}).
In Sec. 5, we derive the stochastic averaging principle in sense of strong convergence for (\ref{1}).
\section{Preliminary results}
To prove the main theorems some preliminary results will be needed.
In this section we gather several technical lemmas.
\subsection{The semigroup $\{S(t)\}_{t\geq 0}$ associated to $-\mathcal{L}$}
\par
According to \cite[P83]{Y1}, the operator $-\mathcal{L}$ is positive, self-adjoint and sectorial on the
domain $\mathcal{D}(-\mathcal{L})=H^{2}(I)\cap H_{0}^{1}(I)$. By spectral theory, we may define
the fractional powers $(-\mathcal{L})^{\alpha}$ of $-\mathcal{L}$ with the domain $\mathcal{D}((-\mathcal{L})^{\alpha})$ for any
$\alpha\in [0,1]$. We know that the semigroup $\{S(t)\}_{t\geq 0}$ generated by the
operator $-\mathcal{L}$ is analytic on $L^{p}(I)$ for all $1\leq p\leq\infty$ and enjoys the
following properties \cite{P1}:
\begin{eqnarray}\label{13}
\begin{array}{l}
\begin{array}{llll}
S(t)(-\mathcal{L})^{\alpha}=(-\mathcal{L})^{\alpha}S(t),~~~~~~~~~~~~~~~~~~~~~~~~~~~~~~~\alpha\geq0,
\\\|(-\mathcal{L})^{\alpha}S(t)\varphi\|_{L^{p}(I)}\leq Ct^{-\alpha}\|\varphi\|_{L^{p}(I)},~~~~~~~~~~~~\alpha\geq0,t\geq0,
\\
\|D^{j}S(t)\varphi\|_{L^{q}(I)}\leq Ct^{-\frac{1}{2}(\frac{1}{p}-\frac{1}{q}+j)}\|\varphi\|_{L^{p}(I)},~~q\geq p\geq1,t\geq0,
\end{array}
\end{array}
\end{eqnarray}
where $D^{j}$ denotes the $j-$th order derivative with respect to the spatial variable.

\subsection{Some useful inequalities}

\begin{lemma}
If ~$a,b\in \mathbb{R}$, $p>0,$ it holds that
\begin{eqnarray*}
\begin{array}{l}
(|a|+|b|)^{p}\leq\left\{
\begin{array}{llll}
|a|^{p}+|b|^{p}~~~~~~~~~~~~~~~0<p\leq 1,
\\2^{p-1}(|a|^{p}+|b|^{p})~~~~~~~~~
p>1.

\end{array}
\right.
\end{array}
\end{eqnarray*}
\end{lemma}

\begin{lemma}\label{L9}(Young inequality)
Let $a, b \in [0, +\infty)$ and $\varepsilon>0$, then we have
\begin{eqnarray*}
ab\leq \varepsilon^{-p}\frac{a^{p}}{p}+\varepsilon^{q}\frac{b^{q}}{q},
\end{eqnarray*}
where $1 < p < \infty,\frac{1}{p}+\frac{1}{q} = 1.$
\end{lemma}
\begin{lemma}\label{L7}
Let $y(t)$ be a nonnegative function, if
\begin{eqnarray*}
\begin{array}{l}
\begin{array}{llll}
y^{\prime}\leq -ay+f,

\end{array}
\end{array}
\end{eqnarray*}
we have
\begin{eqnarray*}
\begin{array}{l}
\begin{array}{llll}
y(t)\leq y(s)e^{-a(t-s)}+\int_{s}^{t}e^{-a(t-\tau)}f(\tau)d\tau.

\end{array}
\end{array}
\end{eqnarray*}
\end{lemma}

\subsection{Some useful estimates}
\par
The following lemmas are very useful in establishing a priori estimate for the slow-fast system.
\begin{lemma}\cite[Lemma 7.2]{L3}\label{L8}
Let $A_{1}$ and $A_{2}$ be two complex-valued numbers and $\sigma\geq \frac{1}{2}$. Then
the following inequality is fulfilled
\begin{eqnarray*}
\begin{array}{l}
\begin{array}{llll}
||A_{1}|^{2\sigma}A_{1}-|A_{2}|^{2\sigma}A_{2}|\leq (4\sigma-1)(|A_{1}|^{2\sigma}+|A_{2}|^{2\sigma})|A_{1}-A_{2}|.
\end{array}
\end{array}
\end{eqnarray*}
\end{lemma}
\begin{remark}
The same result can be found in \cite[P8]{B3}.
\end{remark}

\begin{lemma}\cite[Lemma 7.3]{L3}
Let $A_{1}$ and $A_{2}$ be two complex-valued numbers and $\sigma> 0$. Then
the following inequality is fulfilled
\begin{eqnarray*}
\begin{array}{l}
\begin{array}{llll}
\Re\{(A_{1}-A_{2})(|A_{1}|^{2\sigma}A_{1}-|A_{2}|^{2\sigma}A_{2})\}\geq0.
\end{array}
\end{array}
\end{eqnarray*}
\end{lemma}
Thus we have

\begin{corollary}\label{L1}
For any $A_{1},A_{2}\in \mathbb{C},$ we have
\begin{eqnarray*}
\begin{array}{l}
\begin{array}{llll}
(A_{1}-A_{2},\mathcal{F}(A_{1})-\mathcal{F}(A_{2}))\leq 0,
\\
(A_{1}-A_{2},\mathcal{G}(A_{1})-\mathcal{G}(A_{2}))\leq 0.
\end{array}
\end{array}
\end{eqnarray*}
\end{corollary}

\par
The following lemma is very useful in establishing a priori estimate for the slow-fast system.
\begin{lemma}\label{L5}\cite[Lemma 2.6]{Y1}
If $\sigma>0,|\alpha|<\frac{\sqrt{2\sigma+1}}{\sigma},$ there exists a positive constant $\lambda_{\alpha}$ such that
\begin{eqnarray*}
\begin{array}{l}
\begin{array}{llll}
(-A_{xx},(-1+\alpha i)|A|^{2\sigma}A)+\lambda_{\alpha}\int_{I}|A|^{2\sigma}|A_{x}|^{2}dx\leq 0.
\end{array}
\end{array}
\end{eqnarray*}
In pariculiarty, we have
\begin{eqnarray*}
\begin{array}{l}
\begin{array}{llll}
(-A_{xx},(-1+\alpha i)|A|^{2\sigma}A)\leq 0.
\end{array}
\end{array}
\end{eqnarray*}
\end{lemma}
\begin{remark}
The same results can be found in \cite[Lemma 7.4]{L3}.
\end{remark}

\section{The fast motion equation (\ref{3})}
\par
First, we consider the stochastic Ginzburg-Landau equation, the solution of (\ref{3}) will be denoted by $B^{A,B_{0}}.$
\par
We could have the following property for the solution of (\ref{3}):
\begin{lemma}\label{L6} For $A\in L^{2}(I),$ let $B^{A,X}$ be the solution of
\begin{equation}\label{14}
\begin{array}{l}
\left\{
\begin{array}{llll}
dB=[\mathcal{L}(B)+\mathcal{F}(B)+\mathcal{G}(B)+\eta B+i\kappa A]dt+\sigma_{2}dW_{2}
\\B(0,t)=0=B(1,t)
\\B(x,0)=X(x)
\end{array}
\right.
\end{array}
\begin{array}{lll}
{\rm{in}}~I\times(0,+\infty)\\
{\rm{in}}~(0,+\infty)\\
{\rm{in}}~I.
\end{array}
\end{equation}
\par
1) There exists a positive constant $C$ such that $B^{A,X}$ satisfies:
\begin{eqnarray}\label{22}
\begin{array}{l}
\begin{array}{llll}
\mathbb{E}\|B^{A,X}(t)\|^{2}\leq e^{-2\alpha t}\|X\|^{2}+C(\|A\|^{2}+1),
\\
\mathbb{E}\|B^{A,X}(t)-B^{A,Y}(t)\|^{2}\leq \|X-Y\|^{2}e^{-2\alpha t},

\end{array}
\end{array}
\end{eqnarray}
for $t\geq0.$
\par
2) There is unique invariant measure $\mu^{A}$ for the Markov
semigroup $P_{t}^{A}$ associated with the system (\ref{14}) in $L^{2}(I).$
Moreover, we have
\begin{eqnarray*}
\int_{L^{2}(I)}\|z\|^{2}\mu^{A}(dz)\leq C(1+\|A\|^{2}).
\end{eqnarray*}
\par
3) There exists a positive constant $C$ such that $B^{A,X}$ satifies:
\begin{eqnarray*}
\begin{array}{l}
\begin{array}{llll}
\|\mathbb{E}f(A,B^{A,X})-\bar{f}(A)\|^{2}\leq C(1+\|X\|^{2}+\|A\|^{2})e^{-2\alpha t}

\end{array}
\end{array}
\end{eqnarray*}
for $t\geq0.$

\end{lemma}
\begin{proof}
1) $\bullet$ By applying the generalized It\^{o} formula with $\frac{1}{2}\|B^{A,X}\|^{2},$ we can obtain that
\begin{eqnarray*}
\begin{array}{l}
\begin{array}{llll}
\frac{1}{2}\|B^{A,X}\|^{2}=\frac{1}{2}\|X\|^{2}+\int_{0}^{t}(B^{A,X},\mathcal{L}B^{A,X}+\mathcal{F} (B^{A,X})+\mathcal{G} (B^{A,X})+\eta B^{A,X}+i\kappa A)ds
\\~~~~~~~~~~~~~~+\int_{0}^{t}(B^{A,X},\sigma_{2}dW_{2})+\frac{1}{2}\int_{0}^{t}\|\sigma_{2}\|_{Q_{2}}^{2}ds
\\=\frac{1}{2}\|X\|^{2}-\beta\int_{0}^{t}\|B_{x}^{A,X}\|^{2}ds+\eta\int_{0}^{t}\|B^{A,X}\|^{2}ds+\int_{0}^{t}(B^{A,X},i\kappa A)ds
\\~~~~~~~~~+\int_{0}^{t}(B^{A,X},\mathcal{F} (B^{A,X})+\mathcal{G} (B^{A,X}))ds
+\int_{0}^{t}(B^{A,X},\sigma_{2}dW_{2})+\frac{1}{2}\int_{0}^{t}\|\sigma_{2}\|_{Q_{2}}^{2}ds
.
\end{array}
\end{array}
\end{eqnarray*}
Taking mathematical expectation from both sides of above equation, we have
\begin{eqnarray*}
\begin{array}{l}
\begin{array}{llll}
\mathbb{E}\|B^{A,X}\|^{2}=\|X\|^{2}-2\beta\int_{0}^{t}\mathbb{E}\|B_{x}^{A,X}\|^{2}ds+2\eta\int_{0}^{t}\mathbb{E}\|B^{A,X}\|^{2}ds+2\int_{0}^{t}\mathbb{E}(B^{A,X},i\kappa A)ds
\\~~~~~~~~~+2\int_{0}^{t}\mathbb{E}(B^{A,X},\mathcal{F} (B^{A,X})+\mathcal{G} (B^{A,X}))ds+\int_{0}^{t}\|\sigma_{2}\|_{Q_{2}}^{2}ds
,
\end{array}
\end{array}
\end{eqnarray*}
namely,
\begin{eqnarray*}
\begin{array}{l}
\begin{array}{llll}
\frac{d}{dt}\mathbb{E}\|B^{A,X}\|^{2}
\\=-2\beta\mathbb{E}\|B_{x}^{A,X}\|^{2}+2\eta\mathbb{E}\|B^{A,X}\|^{2}+2\mathbb{E}(B^{A,X},i\kappa A)
+2\mathbb{E}(B^{A,X},\mathcal{F} (B^{A,X})+\mathcal{G} (B^{A,X}))+\|\sigma_{1}\|_{Q_{1}}^{2}
.
\end{array}
\end{array}
\end{eqnarray*}
According to Corollary \ref{L1}, we have
\begin{eqnarray*}
\begin{array}{l}
\begin{array}{llll}
(B^{A,X},\mathcal{F} (B^{A,X})+\mathcal{G} (B^{A,X}))\leq 0,
\end{array}
\end{array}
\end{eqnarray*}
thus,
\begin{eqnarray*}
\begin{array}{l}
\begin{array}{llll}
\frac{d}{dt}\mathbb{E}\|B^{A,X}\|^{2}
\\\leq-2\beta\mathbb{E}\|B_{x}^{A,X}\|^{2}+2\eta\mathbb{E}\|B^{A,X}\|^{2}+\beta\lambda\mathbb{E}\|B^{A,X}\|^{2}+C(\lambda,\beta,\kappa)\|A\|^{2}+\|\sigma_{1}\|_{Q_{1}}^{2}
\\\leq-2\beta\mathbb{E}\|B_{x}^{A,X}\|^{2}+2\eta\mathbb{E}\|B^{A,X}\|^{2}+\beta\mathbb{E}\|B^{A,X}_{x}\|^{2}+C\|A\|^{2}+\|\sigma_{1}\|_{Q_{1}}^{2}
\\=-\beta\mathbb{E}\|B_{x}^{A,X}\|^{2}+2\eta\mathbb{E}\|B^{A,X}\|^{2}+C\|A\|^{2}+\|\sigma_{1}\|_{Q_{1}}^{2}
\\\leq-\beta\lambda\mathbb{E}\|B^{A,X}\|^{2}+2\eta\mathbb{E}\|B^{A,X}\|^{2}+C\|A\|^{2}+\|\sigma_{1}\|_{Q_{1}}^{2}
\\=-(\beta\lambda-2\eta)\mathbb{E}\|B^{A,X}\|^{2}+C\|A\|^{2}+\|\sigma_{1}\|_{Q_{1}}^{2}
\\=-2\alpha\mathbb{E}\|B^{A,X}\|^{2}+C\|A\|^{2}+\|\sigma_{1}\|_{Q_{1}}^{2}
.
\end{array}
\end{array}
\end{eqnarray*}
Hence, by applying Lemma \ref{L7} with $\mathbb{E}\|B^{A,X}\|^{2}$, we have
\begin{eqnarray*}
\begin{array}{l}
\begin{array}{llll}
\mathbb{E}\|B^{A,X}(t)\|^{2}\leq e^{-2\alpha t}\|X\|^{2}+C(\|A\|^{2}+1).
\end{array}
\end{array}
\end{eqnarray*}

\par
$\bullet$
It is easy to see
\begin{eqnarray*}
\begin{array}{l}
\left\{
\begin{array}{llll}
d(B^{A,X}-B^{A,Y})=[\mathcal{L}(B^{A,X}-B^{A,Y})+\mathcal{F}(B^{A,X})-\mathcal{F}(B^{A,Y})
\\~~~~~~~~~~~~~~~+\mathcal{G}(B^{A,X})-\mathcal{G}(B^{A,Y})+\eta (B^{A,X}-B^{A,Y})]dt
\\(B^{A,X}-B^{A,Y})(0,t)=0=(B^{A,X}-B^{A,Y})(1,t)
\\(B^{A,X}-B^{A,Y})(x,0)=X-Y
\end{array}
\right.
\end{array}
\begin{array}{lll}
\\
{\rm{in}}~Q\\
{\rm{in}}~(0,T)\\
{\rm{in}}~I,
\end{array}
\end{eqnarray*}
thus, it follows from the energy method that
\begin{eqnarray*}
\begin{array}{l}
\begin{array}{llll}
\frac{1}{2}\|B^{A,X}-B^{A,Y}\|^{2}

\\=\frac{1}{2}\|X-Y\|^{2}+\int_{0}^{t}(B^{A,X}-B^{A,Y},\mathcal{L}(B^{A,X}-B^{A,Y})+\mathcal{F}(B^{A,X})-\mathcal{F}(B^{A,Y})
\\~~+\mathcal{G}(B^{A,X})-\mathcal{G}(B^{A,Y})+\eta (B^{A,X}-B^{A,Y}))ds

\\=\frac{1}{2}\|X-Y\|^{2}-\beta\int_{0}^{t}\|(B^{A,X}-B^{A,Y})_{x}\|^{2}ds+\eta\int_{0}^{t}\|B^{A,X}-B^{A,Y}\|^{2}ds
\\~~+\int_{0}^{t}(B^{A,X}-B^{A,Y},\mathcal{F}(B^{A,X})-\mathcal{F}(B^{A,Y})+\mathcal{G}(B^{A,X})-\mathcal{G}(B^{A,Y}))ds
,
\end{array}
\end{array}
\end{eqnarray*}
namely,
\begin{eqnarray*}
\begin{array}{l}
\begin{array}{llll}
\frac{d}{dt}\|B^{A,X}-B^{A,Y}\|^{2}

\\=-2\beta\|(B^{A,X}-B^{A,Y})_{x}\|^{2}+2\eta\|B^{A,X}-B^{A,Y}\|^{2}
\\~~+2(B^{A,X}-B^{A,Y},\mathcal{F}(B^{A,X})-\mathcal{F}(B^{A,Y})+\mathcal{G}(B^{A,X})-\mathcal{G}(B^{A,Y}))
.
\end{array}
\end{array}
\end{eqnarray*}
It follows from Lemma \ref{L1}, we have
\begin{eqnarray*}
\begin{array}{l}
\begin{array}{llll}
(B^{A,X}-B^{A,Y},\mathcal{F}(B^{A,X})-\mathcal{F}(B^{A,Y}))\leq 0,
\\
(B^{A,X}-B^{A,Y},\mathcal{G}(B^{A,X})-\mathcal{G}(B^{A,Y}))\leq 0.
\end{array}
\end{array}
\end{eqnarray*}
Thus, we have
\begin{eqnarray*}
\begin{array}{l}
\begin{array}{llll}
\frac{d}{dt}\|B^{A,X}-B^{A,Y}\|^{2}
\\\leq-2\beta\|(B^{A,X}-B^{A,Y})_{x}\|^{2}+2\eta\|B^{A,X}-B^{A,Y}\|^{2}
\\\leq-(\beta\lambda-2\eta)\|B^{A,X}-B^{A,Y}\|^{2}
\\=-2\alpha\|B^{A,X}-B^{A,Y}\|^{2}
,
\end{array}
\end{array}
\end{eqnarray*}
this yields
\begin{eqnarray*}
\begin{array}{l}
\begin{array}{llll}
\|B^{A,X}-B^{A,Y}\|^{2}\leq \|X-Y\|^{2}e^{-2\alpha t}
.
\end{array}
\end{array}
\end{eqnarray*}
Thus, we have
\begin{eqnarray*}
\begin{array}{l}
\begin{array}{llll}
\mathbb{E}\|B^{A,X}-B^{A,Y}\|^{2}\leq \|X-Y\|^{2}e^{-2\alpha t}
.
\end{array}
\end{array}
\end{eqnarray*}
\par
2) (\ref{22}) imply for any $A\in L^{2}(I)$ that there is unique invariant measure $\mu^{A}$ for the Markov
semigroup $P_{t}^{A}$ associated with the system (\ref{14}) in $L^{2}(I)$ such that
\begin{eqnarray*}
\begin{array}{l}
\begin{array}{llll}
\int_{L^{2}(I)}P_{t}^{A}\varphi d\mu^{A}=\int_{L^{2}(I)}\varphi d\mu^{A},~~t\geq0
\end{array}
\end{array}
\end{eqnarray*}
for any $\varphi\in B_{b}(L^{2}(I))$ the space of bounded functions on $L^{2}(I).$
\par
Then by repeating the standard argument as in \cite[Proposition 4.2]{C3} and \cite[Lemma 3.4]{C1}, the invariant
measure satisfies
\begin{eqnarray*}
\int_{L^{2}(I)}\|z\|^{2}\mu^{A}(dz)\leq C(1+\|A\|^{2}).
\end{eqnarray*}
\par
3) According to the invariant property of $\mu^{A},$ (2) and (\ref{22}), we have
\begin{eqnarray*}
\begin{array}{l}
\begin{array}{llll}
\|\mathbb{E}f(A,B^{A,X})-\bar{f}(A)\|^{2}
\\=
\|\mathbb{E}f(A,B^{A,X})-\int_{L^{2}(I)}f(A,Y)\mu ^{A}(dY)\|^{2}

\\=
\|\mathbb{E}f(A,B^{A,X})-\mathbb{E}\int_{L^{2}(I)}f(A,B^{A,Y})\mu^{A}(dY)\|^{2}
\\=
\|\int_{L^{2}(I)}\mathbb{E}[f(A,B^{A,X})-f(A,B^{A,Y})]\mu^{A}(dY)\|^{2}
\\\leq
C\int_{L^{2}(I)}\mathbb{E}\|B^{A,X}-B^{A,Y}\|^{2}\mu^{A}(dY)
\\\leq C\int_{L^{2}(I)}\|X-Y\|^{2}e^{-2\alpha t}\mu^{A}(dY)
\\\leq C(1+\|X\|^{2}+\|A\|^{2})e^{-2\alpha t}
.
\end{array}
\end{array}
\end{eqnarray*}

\end{proof}

\section{Well-posedness and a priori estimate for the slow-fast system (\ref{1}) and averaged equation (\ref{8})}
\par
We first establish the well-posedness for the slow-fast system (\ref{1}).
\par
Since nonlinear terms $\mathcal{F}(A),\mathcal{G}(A),\mathcal{F}(B)$ and $\mathcal{G}(B)$ are not Lipschitz continuous,
we will use a truncation argument which will lead to a local existence result. Then via some a priori estimates we
obtain that the solution is also global.
\subsection{Well-posedness and a priori estimate for the slow-fast system (\ref{1})}

\par
The proof of well-posedness for the slow-fast system (\ref{1}) is divided into
several steps.
\subsubsection{ Local existence}
We can establish the local well-posedness for the slow-fast system (\ref{1}) in $X_{p,T}(p\geq 1).$
\begin{lemma}\label{P5}
For any $(A_{0},B_{0})\in H_{0}^{1}(I)\times H_{0}^{1}(I)$ and $p\geq1,$ $\varepsilon\in(0,1)$ (\ref{1}) admits a unique mild solution $(A^{\varepsilon},B^{\varepsilon})\in X_{p,\tau_{\infty}},$ where $\tau_{\infty}$ is
stopping time for $p.$ Moreover, if $\tau_{\infty}<+\infty,$ then $\mathbb{P}-$a.s.
\begin{eqnarray*}
\begin{array}{l}
\begin{array}{llll}
\limsup\limits_{t\rightarrow \tau_{\infty}}\|(A^{\varepsilon},B^{\varepsilon})\|_{Y_{t}}=+\infty.
\end{array}
\end{array}
\end{eqnarray*}
\end{lemma}
\begin{proof}

Inspired from \cite{L1}, let $\rho\in C^{\infty}_{0}(\mathbb{R})$ be a cut-off function such that $\rho(r)=1$ for $r\in[0,1]$ and $\rho(r)=0$ for $r\geq 2.$ For any $R>0,y\in X_{p,t}$ and $t\in [0,T],$ we set
\begin{eqnarray*}
\begin{array}{l}
\begin{array}{llll}
\rho_{R}(y)(t)=\rho(\frac{\|y\|_{C([0,t];H^{1}(I))}}{R}).
\end{array}
\end{array}
\end{eqnarray*}
The truncated equation corresponding to (\ref{1}) is the
following stochastic partial differential equation:
\begin{equation*}
\begin{array}{l}
\left\{
\begin{array}{llll}
dA^{\varepsilon}=[\mathcal{L}(A^{\varepsilon})+\rho_{R}(A^{\varepsilon})\mathcal{F}(A^{\varepsilon})+\rho_{R}(A^{\varepsilon})\mathcal{G}(A^{\varepsilon})+f(A^{\varepsilon}, B^{\varepsilon})]dt+\sigma_{1}dW_{1}
\\dB^{\varepsilon}=\frac{1}{\varepsilon}[\mathcal{L}(B^{\varepsilon})+\rho_{R}(B^{\varepsilon})\mathcal{F}(B^{\varepsilon})+\rho_{R}(B^{\varepsilon})\mathcal{G}(B^{\varepsilon})+g(A^{\varepsilon}, B^{\varepsilon})]dt+\frac{1}{\sqrt{\varepsilon}}\sigma_{2}dW_{2}
\\A^{\varepsilon}(0,t)=0=A^{\varepsilon}(1,t)
\\B^{\varepsilon}(0,t)=0=B^{\varepsilon}(1,t)
\\A^{\varepsilon}(x,0)=A_{0}(x)
\\B^{\varepsilon}(x,0)=B_{0}(x)
\end{array}
\right.
\end{array}
\begin{array}{lll}
{\rm{in}}~Q,\\
{\rm{in}}~Q,\\
{\rm{in}}~(0,T),\\
{\rm{in}}~(0,T),\\
{\rm{in}}~I,\\
{\rm{in}}~I.
\end{array}
\end{equation*}
\par
In the proof of Lemma \ref{P5}, we will take
\begin{equation*}
\begin{array}{l}
\begin{array}{llll}
\varepsilon=1
\end{array}
\end{array}
\end{equation*}
for the sake of simplicity. All the results
can be extended without difficulty to the general case.\
Thus, we consider the following system
\begin{equation*}
\begin{array}{l}
\left\{
\begin{array}{llll}
dA=[\mathcal{L}(A)+\rho_{R}(A)\mathcal{F}(A)+\rho_{R}(A)\mathcal{G}(A)+f(A, B)]dt+\sigma_{1}dW_{1}
\\dB=[\mathcal{L}(B)+\rho_{R}(B)\mathcal{F}(B)+\rho_{R}(B)\mathcal{G}(B)+g(A, B)]dt+\sigma_{2}dW_{2}
\\A(0,t)=0=A(1,t)
\\B(0,t)=0=B(1,t)
\\A(x,0)=A_{0}(x)
\\B(x,0)=B_{0}(x)
\end{array}
\right.
\end{array}
\begin{array}{lll}
{\rm{in}}~Q,\\
{\rm{in}}~Q,\\
{\rm{in}}~(0,T),\\
{\rm{in}}~(0,T),\\
{\rm{in}}~I,\\
{\rm{in}}~I.
\end{array}
\end{equation*}
\par
We define
\begin{eqnarray*}
\begin{array}{l}
\begin{array}{llll}
\Phi_{R}(A,B)
\\=
\left(
\begin{array}{c}\Phi_{R}^{1}(A,B)
\\\Phi_{R}^{2}(A,B)
\end{array}\right)

\\=\left(\begin{array}{c}S(t)A_{0}+\int_{0}^{t}S(t-s)(\rho_{R}(A)\mathcal{F}(A)+\rho_{R}(A)\mathcal{G}(A)+f(A, B)) (s) ds+\int_{0}^{t}S(t-s)\sigma_{1}dW_{1}
\\S(t)B_{0}+\int_{0}^{t}S(t-s)(\rho_{R}(B)\mathcal{F}(B)+\rho_{R}(B)\mathcal{G}(B)+g(A, B))(s)ds+\int_{0}^{t}S(t-s)\sigma_{2}dW_{2}
\end{array}
\right).
\end{array}
\end{array}
\end{eqnarray*}
\par
$\bullet$
It is easy to see the operator $\Phi_{R}(A,B)$ maps $X_{p,T_{0}}$ into itself.
\par
$\bullet$ The estimates of
\begin{eqnarray*}
\begin{array}{l}
\begin{array}{llll}
\mathbb{E}\sup_{0\leq t\leq T_{0}}\|(\Phi_{R}^{1}(A_{1}, B_{1})-\Phi_{R}^{1}(A_{2}, B_{2}))(t)\|_{H^{1}}^{p},
\\
\mathbb{E}\sup_{0\leq t\leq T_{0}}\|(\Phi_{R}^{2}(A_{1}, B_{1})-\Phi_{R}^{2}(A_{2}, B_{2}))(t)\|_{H^{1}}^{p}.
\end{array}
\end{array}
\end{eqnarray*}
\par
Indeed, due to \cite[P84]{Y1}, we have
\begin{eqnarray*}
\begin{array}{l}
\begin{array}{llll}
\|\rho_{R}(A_{1})|A_{1}|^{2\sigma}A_{1}-\rho_{R}(A_{2})|A_{2}|^{2\sigma}A_{2}\|\leq CR^{2\sigma}\|A_{1}-A_{2}\|_{H^{1}}.
\end{array}
\end{array}
\end{eqnarray*}
By taking $p=q=2,j=1$ in the third inequality of (\ref{13}), we have
\begin{eqnarray}\label{16}
\begin{array}{l}
\begin{array}{llll}
\mathbb{E}\sup_{0\leq t\leq T_{0}}\|\int_{0}^{t}S(t-s)(\rho_{R}(A_{1})\mathcal{F}(A_{1})-\rho_{R}(A_{2})\mathcal{F}(A_{2}))(s)ds\|_{H^{1}}^{p}
\\\leq C\mathbb{E}\sup_{0\leq t\leq T_{0}}(\int_{0}^{t}(t-s)^{-\frac{1}{2}}\|(\rho_{R}(A_{1})\mathcal{F}(A_{1})-\rho_{R}(A_{2})\mathcal{F}(A_{2}))(s)\|ds)^{p}
\\\leq C\mathbb{E}\sup_{0\leq t\leq T_{0}}(\int_{0}^{t}(t-s)^{-\frac{1}{2}}R^{2}\|(A_{1}-A_{2})(s)\|_{H^{1}}ds)^{p}
\\\leq C R^{2p}\sup_{0\leq t\leq T_{0}}(\int_{0}^{t}(t-s)^{-\frac{1}{2}}ds)^{p}\mathbb{E}\sup_{0\leq t\leq T_{0}}\|(A_{1}-A_{2})(t)\|_{H^{1}}^{p}
\\\leq C R^{2p}T_{0}^{\frac{p}{2}}\mathbb{E}\sup_{0\leq t\leq T_{0}}\|(A_{1}-A_{2})(t)\|_{H^{1}}^{p},
\end{array}
\end{array}
\end{eqnarray}

\begin{eqnarray}\label{17}
\begin{array}{l}
\begin{array}{llll}
\mathbb{E}\sup_{0\leq t\leq T_{0}}\|\int_{0}^{t}S(t-s)(\rho_{R}(A_{1})\mathcal{G}(A_{1})-\rho_{R}(A_{2})\mathcal{G}(A_{2}))(s)ds\|_{H^{1}}^{p}
\\\leq C\mathbb{E}\sup_{0\leq t\leq T_{0}}(\int_{0}^{t}(t-s)^{-\frac{1}{2}}\|(\rho_{R}(A_{1})\mathcal{G}(A_{1})-\rho_{R}(A_{2})\mathcal{G}(A_{2}))(s)\|ds)^{p}
\\\leq C\mathbb{E}\sup_{0\leq t\leq T_{0}}(\int_{0}^{t}(t-s)^{-\frac{1}{2}}R^{4}\|(A_{1}-A_{2})(s)\|_{H^{1}}ds)^{p}
\\\leq C R^{4p}\sup_{0\leq t\leq T_{0}}(\int_{0}^{t}(t-s)^{-\frac{1}{2}}ds)^{p}\mathbb{E}\sup_{0\leq t\leq T_{0}}\|(A_{1}-A_{2})(t)\|_{H^{1}}^{p}
\\\leq C R^{4p}T_{0}^{\frac{p}{2}}\mathbb{E}\sup_{0\leq t\leq T_{0}}\|(A_{1}-A_{2})(t)\|_{H^{1}}^{p},
\end{array}
\end{array}
\end{eqnarray}

and
\begin{eqnarray}\label{18}
\begin{array}{l}
\begin{array}{llll}
\mathbb{E}\sup_{0\leq t\leq T_{0}}\|\int_{0}^{t}S(t-s)(f(A_{1}, B_{1})-f(A_{2}, B_{2}))(s)ds\|_{H^{1}}^{p}
\\\leq \mathbb{E}\sup_{0\leq t\leq T_{0}}(\int_{0}^{t}\|S(t-s)(f(A_{1}, B_{1})-f(A_{2}, B_{2}))(s)\|_{H^{1}}ds)^{p}
\\\leq C\mathbb{E}\sup_{0\leq t\leq T_{0}}(\int_{0}^{t}(t-s)^{-\frac{1}{2}}\|(f(A_{1}, B_{1})-f(A_{2}, B_{2}))(s)\|ds)^{p}
\\\leq C\mathbb{E}\sup_{0\leq t\leq T_{0}}(\int_{0}^{t}(t-s)^{-\frac{1}{2}}(\|(A_{1}-A_{2})(s)\|+\|(B_{1}-B_{2})(s)\|)ds)^{p}
\\\leq C \sup_{0\leq t\leq T_{0}}(\int_{0}^{t}(t-s)^{-\frac{1}{2}}ds)^{p}(\mathbb{E}\sup_{0\leq t\leq T_{0}}\|(A_{1}-A_{2})(t)\|^{p}+\mathbb{E}\sup_{0\leq t\leq T_{0}}\|(B_{1}-B_{2})(t)\|^{p})
\\\leq C T_{0}^{\frac{p}{2}}(\mathbb{E}\sup_{0\leq t\leq T_{0}}\|(A_{1}-A_{2})(t)\|^{p}+\mathbb{E}\sup_{0\leq t\leq T_{0}}\|(B_{1}-B_{2})(t)\|^{p}).
\end{array}
\end{array}
\end{eqnarray}
Finally, collecting the above estimates (\ref{16})-(\ref{18}), we get
\begin{eqnarray}\label{19}
\begin{array}{l}
\begin{array}{llll}
\mathbb{E}\sup_{0\leq t\leq T_{0}}\|(\Phi_{R}^{1}(A_{1}, B_{1})-\Phi_{R}^{1}(A_{2}, B_{2}))(t)\|_{H^{1}}^{p}
\\\leq C(R^{2p}T_{0}^{\frac{p}{2}}+R^{4p}T_{0}^{\frac{p}{2}}+T_{0}^{\frac{p}{2}})
(\mathbb{E}\sup_{0\leq t\leq T_{0}}\|(A_{1}-A_{2})(t)\|_{H^{1}}^{p}+\mathbb{E}\sup_{0\leq t\leq T_{0}}\|(B_{1}-B_{2})(t)\|_{H^{1}}^{p}).
\end{array}
\end{array}
\end{eqnarray}

\par
By the same method, we have
\begin{eqnarray}\label{20}
\begin{array}{l}
\begin{array}{llll}
\mathbb{E}\sup_{0\leq t\leq T_{0}}\|(\Phi_{R}^{2}(A_{1}, B_{1})-\Phi_{R}^{2}(A_{2}, B_{2}))(t)\|_{H^{1}}
^{p}
\\\leq C(R^{2p}T_{0}^{\frac{p}{2}}+R^{4p}T_{0}^{\frac{p}{2}}+T_{0}^{\frac{p}{2}})
(\mathbb{E}\sup_{0\leq t\leq T_{0}}\|(A_{1}-A_{2})(t)\|_{H^{1}}^{p}+\mathbb{E}\sup_{0\leq t\leq T_{0}}\|(B_{1}-B_{2})(t)\|_{H^{1}}^{p}).
\end{array}
\end{array}
\end{eqnarray}
\par
It follows from (\ref{19}) and (\ref{20}) that
\begin{eqnarray*}
\begin{array}{l}
\begin{array}{llll}
\mathbb{E}\sup_{0\leq t\leq T_{0}}\|(\Phi_{R}^{1}(A_{1}, B_{1})-\Phi_{R}^{1}(A_{2}, B_{2}))(t)\|_{H^{1}}^{p}
+\mathbb{E}\sup_{0\leq t\leq T_{0}}\|(\Phi_{R}^{2}(A_{1}, B_{1})-\Phi_{R}^{2}(A_{2}, B_{2}))(t)\|_{H^{1}}^{p}
\\\leq C(R^{2p}T_{0}^{\frac{p}{2}}+R^{4p}T_{0}^{\frac{p}{2}}+T_{0}^{\frac{p}{2}})
(\mathbb{E}\sup_{0\leq t\leq T_{0}}\|(A_{1}-A_{2})(t)\|_{H^{1}}^{p}+\mathbb{E}\sup_{0\leq t\leq T_{0}}\|(B_{1}-B_{2})(t)\|_{H^{1}}^{p}),
\end{array}
\end{array}
\end{eqnarray*}
namely, we have
\begin{eqnarray}\label{21}
\begin{array}{l}
\begin{array}{llll}
\|\Phi_{R}(A_{1}, B_{1})-\Phi_{R}(A_{2}, B_{2})\|_{X_{p,T_{0}}}
\\\leq C(R^{2}T_{0}^{\frac{1}{2}}+R^{4}T_{0}^{\frac{1}{2}}+T_{0}^{\frac{1}{2}})\|(A_{1},B_{1})-(A_{2},B_{2})\|_{X_{p,T_{0}}}.
\end{array}
\end{array}
\end{eqnarray}
\par
$\bullet$
For a sufficiently small $T_{0},$ is $\Phi_{R}(A,B)$ a contraction mapping on $X_{p,T_{0}}.$
\par
Hence, by applying the Banach contraction principle, $\Phi_{R}(A,B)$ has a unique fixed point in $X_{p,T_{0}},$
which is the unique local solution to (\ref{1}) on the interval
$[0,T_{0}].$ Since $T_{0}$ does not depend on the initial value $(A_{0},B_{0}),$ this
solution may be extended to the whole interval $[0,T].$
\par
We denote by $(A_{R},B_{R})$ this unique mild solution and let
\begin{eqnarray*}
\begin{array}{l}
\begin{array}{llll}
\tau_{R}=\inf\{t\geq0:\|(A_{R},B_{R})\|_{X_{p,t}}\geq R\},
\end{array}
\end{array}
\end{eqnarray*}
with the usual convention that $\inf \emptyset=\infty.$
\par
Since $R_{1}\leq R_{2},$ $\tau_{{R}_{1}}\leq \tau_{{R}_{2}},$ we can put $\tau_{\infty}=\lim\limits_{R\rightarrow +\infty}\tau_{R}.$
We define a local solution to (\ref{1}) as follows
\begin{eqnarray*}
\begin{array}{l}
\begin{array}{llll}
A(t)=A_{R}(t),~\forall t\in [0,\tau_{R}],
\\B(t)=B_{R}(t),~\forall t\in [0,\tau_{R}].
\end{array}
\end{array}
\end{eqnarray*}
Indeed, for any $t\in [0,\tau_{R_{1}}\wedge\tau_{R_{2}}]$
\begin{eqnarray*}
\begin{array}{l}
\begin{array}{llll}
~~~~A_{R_{1}}(t)-A_{R_{2}}(t)
\\=\int_{0}^{t}S(t-s)(\rho_{R_{1}}(A_{R_{1}})\mathcal{F}(A_{R_{1}})-\rho_{R_{2}}(A_{R_{2}})\mathcal{F}(A_{R_{2}})
+\rho_{R_{1}}(A_{R_{1}})\mathcal{G}(A_{R_{1}})-\rho_{R_{2}}(A_{R_{2}})\mathcal{G}(A_{R_{2}})
\\~~~~~~~~~~~~~~~~~~~~~~~~~~~~+f(A_{R_{1}}, B_{R_{1}})-f(A_{R_{2}}, B_{R_{2}})) (s) ds,
\\
~~~~B_{R_{1}}(t)-B_{R_{2}}(t)
\\=\int_{0}^{t}S(t-s)(\rho_{R_{1}}(B_{R_{1}})\mathcal{F}(B_{R_{1}})-\rho_{R_{2}}(B_{R_{2}})\mathcal{F}(B_{R_{2}})
+\rho_{R_{1}}(B_{R_{1}})\mathcal{G}(B_{R_{1}})-\rho_{R_{2}}(B_{R_{2}})\mathcal{G}(B_{R_{2}})
\\~~~~~~~~~~~~~~~~~~~~~~~~~~~~+g(A_{R_{1}}, B_{R_{1}})-g(A_{R_{2}}, B_{R_{2}})) (s) ds.
\end{array}
\end{array}
\end{eqnarray*}
Proceeding as in the proof of (\ref{21}), we can obtain
\begin{eqnarray*}
\begin{array}{l}
\begin{array}{llll}
\|(A_{R_{1}},B_{R_{1}})-(A_{R_{2}},B_{R_{2}})\|_{X_{p,t}}
\\\leq C(t)\|(A_{R_{1}},B_{R_{1}})-(A_{R_{2}},B_{R_{2}})\|_{X_{p,t}},
\end{array}
\end{array}
\end{eqnarray*}
where $C(t)$ is a monotonically increasing function and $C(0)=0.$
If we take $t$ sufficiently small, we can obtain
\begin{eqnarray*}
\begin{array}{l}
\begin{array}{llll}
A_{R_{1}}(t)=A_{R_{2}}(t),
\\B_{R_{1}}(t)=B_{R_{2}}(t).
\end{array}
\end{array}
\end{eqnarray*}
Repeating the same argument for the interval $[t,2t]$ and so on yields
\begin{eqnarray*}
\begin{array}{l}
\begin{array}{llll}
A_{R_{1}}(t)=B_{R_{2}}(t),
\\A_{R_{1}}(t)=B_{R_{2}}(t),
\end{array}
\end{array}
\end{eqnarray*}
 for the whole interval $[0,\tau].$
According to this, we can know that the above definition of local solution to (\ref{1}) is well defined.
\par
If $\tau_{\infty}<+\infty,$ the definition of $(A,B)$ yields $P-$a.s.
\begin{eqnarray*}
\begin{array}{l}
\begin{array}{llll}
\lim\limits_{t\rightarrow \tau_{\infty}}\|(A,B)\|_{X_{p,t}}=+\infty,
\end{array}
\end{array}
\end{eqnarray*}
which shows that $(A,B)$ is a unique local solution to (\ref{1}) on the
interval $[0,\tau_{\infty}).$
\par
This completes the proof of Lemma \ref{P5}.
\end{proof}

\subsubsection{Some energy inequalities for the slow-fast system (\ref{1})}
\par
Next, we will exploit some energy inequalities for the slow-fast system (\ref{1}).
\begin{lemma}\label{L10}
Let $\xi=\inf\{\tau_{\infty},T\}.$ If $A_{0},B_{0}\in H^{1}_{0}(I),$ for $\varepsilon\in(0,1),$ $(A^{\varepsilon},B^{\varepsilon})$ is the unique solution to (\ref{1}),  then there exists a constant $C$ such that the solutions $(A^{\varepsilon},B^{\varepsilon})$ satisfy
\begin{eqnarray*}
\begin{array}{l}
\begin{array}{llll}
\sup\limits_{\varepsilon\in (0,1)}\mathbb{E}\sup\limits_{t\in[0,\xi]}\|A^{\varepsilon}(t)\|_{H^{1}}^{2}\leq C,
\\
\mathbb{E}\sup\limits_{t\in[0,\xi]}\|B^{\varepsilon}(t)\|_{H^{1}}^{2}\leq \frac{C}{\varepsilon},
\\
\sup\limits_{\varepsilon\in (0,1)}\mathbb{E}\int_{0}^{\xi}\|A^{\varepsilon}_{xx}\|^{2}dt
\leq C,
\\
\sup\limits_{\varepsilon\in (0,1)}\mathbb{E}\int_{0}^{\xi}\|B^{\varepsilon}_{xx}\|^{2}dt
\leq C,
\end{array}
\end{array}
\end{eqnarray*}
where $C$ is dependent of $T,A_{0},B_{0}$ but independent of $\varepsilon\in (0,1).$

\end{lemma}

\begin{proof}
The proof of Lemma \ref{L10} is divided into
several steps. Here, the method of the proof is inspired from \cite{D2,F1,F2,F3,F4,Y1}.
\par $\bullet$ The estimates of $\sup\limits_{t\in[0,\xi]}\mathbb{E}\|A^{\varepsilon}\|_{H^{1}}^{2}$ and $\sup\limits_{t\in[0,\xi]}\mathbb{E}\|B^{\varepsilon}\|_{H^{1}}^{2}.$

\par
$\star$ Indeed, we apply the generalized It\^{o} formula (see \cite{Y1,C4,D1,P2}) with $\|A^{\varepsilon}_{x}\|^{2}$ and obtain that
\begin{eqnarray}\label{25}

\end{eqnarray}
namely, we have
\begin{eqnarray}\label{2}
%
\end{eqnarray}
by taking mathematical expectation from both sides of above equation, we have
\begin{eqnarray*}
%
\end{eqnarray*}
It is easy to see
\begin{eqnarray*}
%
\end{eqnarray*}
According to Lemma \ref{L5}, we have
\begin{eqnarray*}
%
\end{eqnarray*}
thus, it holds that
\begin{eqnarray*}
%
\end{eqnarray*}
it follows from Gronwall inequality that

\begin{eqnarray}\label{6}
%
\end{eqnarray}
\par
$\star$ For $B^{\varepsilon},$ we apply the generalized It\^{o} formula with $\frac{1}{2}\|B^{\varepsilon}_{x}\|^{2}$ and obtain that
\begin{eqnarray}\label{27}
%
\end{eqnarray}
namely, we have
\begin{eqnarray}\label{24}
%
\end{eqnarray}
by taking mathematical expectation from both sides of above equation, we have
\begin{eqnarray}\label{10}
%
\end{eqnarray}
it is easy to see
\begin{eqnarray*}
%
\end{eqnarray*}

According to Lemma \ref{L5}, we have
\begin{eqnarray*}
%
\end{eqnarray*}
it holds that
\begin{eqnarray*}
%
\end{eqnarray*}
hence, by applying Lemma \ref{L7} with $\mathbb{E}\|B^{\varepsilon}_{x}\|^{2},$ we have
\begin{eqnarray*}
%
\end{eqnarray*}
Combining this and (\ref{6}), we have
\begin{eqnarray*}
%
\end{eqnarray*}
it follows from Gronwall inequality that
\begin{eqnarray}\label{5}
%
\end{eqnarray}
By replacing this estimate in (\ref{6}), we have
\begin{eqnarray}\label{7}
%
\end{eqnarray}
\par
$\bullet$ The estimates of $\mathbb{E}\sup\limits_{t\in[0,\xi]}\|A^{\varepsilon}\|_{H^{1}}^{2}$ and $\mathbb{E}\sup\limits_{t\in[0,\xi]}\|B^{\varepsilon}\|_{H^{1}}^{2}.$
\par
$\star$ It follows from (\ref{2}) and Lemma \ref{L5} that
\begin{eqnarray*}
%
\end{eqnarray*}
by the Cauchy inequality, we have
\begin{eqnarray*}
%
\end{eqnarray*}
in view of the Burkholder-Davis-Gundy inequality, it holds that
\begin{eqnarray*}
%
\end{eqnarray*}
Thus, we have
\begin{eqnarray*}
%
\end{eqnarray*}
according to (\ref{5}) and (\ref{7}), it holds that
\begin{eqnarray*}
%
\end{eqnarray*}

\par
$\star$ It follows from (\ref{24}) and Lemma \ref{L5} that
\begin{eqnarray*}
%
\end{eqnarray*}
by the Cauchy inequality, we have
\begin{eqnarray*}
%
\end{eqnarray*}
in view of the Burkholder-Davis-Gundy inequality, it holds that
\begin{eqnarray*}
%
\end{eqnarray*}
in the last inequality, we have used $\varepsilon\in (0,1).$
\par
Thus, we have
\begin{eqnarray*}
%
\end{eqnarray*}
according to (\ref{5}) and (\ref{7}), it holds that
\begin{eqnarray*}
%
\end{eqnarray*}
\par
$\bullet$ The estimates of $\mathbb{E}\int_{0}^{\xi}\|A^{\varepsilon}_{xx}\|^{2}dt$ and $\mathbb{E}\int_{0}^{\xi}\|B^{\varepsilon}_{xx}\|^{2}dt.$
\par
$\star$ Indeed, according to (\ref{2}) and Lemma \ref{L5}, we have
\begin{eqnarray*}
%
\end{eqnarray*}
by taking mathematical expectation from both sides of above equation, we have
\begin{eqnarray*}
%
\end{eqnarray*}
It follows from (\ref{5}) and (\ref{7}) that
\begin{eqnarray}\label{26}
%
\end{eqnarray}
\par
$\star$
It follows from (\ref{10}) and Lemma \ref{L5} that
\begin{eqnarray*}
%
\end{eqnarray*}
thus, we have
\begin{eqnarray*}
%
\end{eqnarray*}
namely, it holds that
\begin{eqnarray}\label{23}
%
\end{eqnarray}

\end{proof}

\subsubsection{Proof of Theorem \ref{Th2}}
\par
Now, we prove Theorem \ref{Th2}.
\begin{proof}[Proof of Theorem \ref{Th2}]
For any fixed $\varepsilon\in (0,1),$ by the Chebyshev inequality, Lemma \ref{L10} and the definition of $(A^{\varepsilon},B^{\varepsilon}),$ we have
\begin{eqnarray*}
%
\end{eqnarray*}
this shows that
\begin{eqnarray*}
%
\end{eqnarray*}
namely, $\tau_{\infty}=+\infty$ P-a.s.
\end{proof}

\subsubsection{Some a priori estimates for the slow-fast system (\ref{1})}
\par
Next, we establish some a priori estimates for the slow-fast system (\ref{1}).
\begin{proposition}\label{P1}
If $A_{0},B_{0}\in H^{1}_{0}(I),$ for $\varepsilon\in(0,1),$ $(A^{\varepsilon},B^{\varepsilon})$ is the unique solution to (\ref{1}),  then for any $p>0,$ there exists a constant $C_{1}$ such that the solutions $(A^{\varepsilon},B^{\varepsilon})$ satisfy
\begin{eqnarray*}
%
\end{array}
\end{eqnarray*}
where $C$ is dependent of $p,T,A_{0},B_{0}$ but independent of $\varepsilon\in (0,1).$
\end{proposition}

\begin{proof}
The proof of Proposition \ref{P1} is divided into
several steps.
It is also suffice to prove Proposition \ref{P1} holds when $p$ is large enough. Here, the method of the proof is inspired from \cite{D2,F1,F2,F3,F4}.

\par
$\bullet$ The estimates of $\sup\limits_{0\leq t\leq T}\mathbb{E}\|A^{\varepsilon}(t)\|^{2p}$ and $\sup\limits_{0\leq t\leq T}\mathbb{E}\|B^{\varepsilon}(t)\|^{2p}.$
\par
$\star$ Indeed, we apply the generalized It\^{o} formula (see \cite{C4,D1,P2}) with $\|A^{\varepsilon}\|^{2p}$ and obtain that
\begin{eqnarray*}

\end{eqnarray*}
by taking mathematical expectation from both sides of above equation, we have
\begin{eqnarray*}
%
\end{eqnarray*}
it is easy to see
\begin{eqnarray*}
%
\end{eqnarray*}
thus, we have
\begin{eqnarray*}
%
\end{eqnarray*}
It follows from the Gronwall inequality that
\begin{eqnarray*}
%
\end{eqnarray*}

\par
$\star$ We apply the generalized It\^{o} formula with $\|B^{\varepsilon}\|^{2p}$ and obtain that
\begin{eqnarray*}
%
\end{eqnarray*}
by taking mathematical expectation from both sides of above equation, we have
\begin{eqnarray*}
%
\end{eqnarray*}
it is easy to see
\begin{eqnarray*}
%
\end{eqnarray*}
thus, we have
\begin{eqnarray*}
%
\end{eqnarray*}

By using the Young inequality(see Lemma \ref{L9}), we have
\begin{eqnarray*}
%
\end{eqnarray*}
Hence, by applying Lemma \ref{L7} with $\mathbb{E}\|B^{\varepsilon}(t)\|^{2p}$, we have
\begin{eqnarray*}
%
\end{eqnarray*}

Thus, by the same method in the above, we have
\begin{eqnarray*}
%
\end{eqnarray*}
thus, it follows from Gronwall inequality that
\begin{eqnarray*}
%
\end{eqnarray*}
Moreover, we have
\begin{eqnarray*}
%
\end{eqnarray*}
\par
$\bullet$ The estimates of $\sup\limits_{0\leq t\leq T}\mathbb{E}\|A^{\varepsilon}(t)\|_{H^{1}}^{2p}$ and $\mathbb{E}\int_{0}^{T}\|B^{\varepsilon}_{x}\|^{2p}dt.$
\par
$\star$
Indeed, it follows from (\ref{25}) that
\begin{eqnarray*}
%
\end{eqnarray*}
then, we apply the generalized It\^{o} formula (see \cite{C4,D1,P2}) with $\|A^{\varepsilon}_{x}\|^{2p}$ and obtain that
\begin{eqnarray*}
%
\end{eqnarray*}
thus, we have
\begin{eqnarray*}
%
\end{eqnarray*}
According to Lemma \ref{L5}, we have
\begin{eqnarray*}
%
\end{eqnarray*}
thus, it holds that
\begin{eqnarray*}
%
\end{eqnarray*}
It follows from the Young inequality(see Lemma \ref{L9}) that
\begin{eqnarray}\label{29}
%
\end{eqnarray}
by taking mathematical expectation from both sides of above equation, we have
\begin{eqnarray*}
%
\end{eqnarray*}
If we take $0<\rho<<1,$ we have
\begin{eqnarray*}
%
\end{eqnarray*}
It follows from the Young inequality(see Lemma \ref{L9}) and (\ref{26}) that
\begin{eqnarray*}
%
\end{eqnarray*}
hence, by Gronwall inequality, we have
\begin{eqnarray}\label{12}
%
\end{eqnarray}
\par
$\star$ Indeed, it follows from (\ref{27}) that
\begin{eqnarray*}
%
\end{eqnarray*}
then, we apply the generalized It\^{o} formula (see \cite{C4,D1,P2}) with $\|B^{\varepsilon}_{x}\|^{2p}$ and obtain that
\begin{eqnarray*}
%
\end{eqnarray*}
According to Lemma \ref{L5}, we have
\begin{eqnarray*}
%
\end{eqnarray*}
thus, it holds that

\begin{eqnarray*}
%
\end{eqnarray*}
it follows from the Young inequality(see Lemma \ref{L9}) that
\begin{eqnarray*}
%
\end{eqnarray*}
By taking mathematical expectation from both sides of above equation, we have
\begin{eqnarray*}
%
\end{eqnarray*}
It follows from (\ref{23}) that
\begin{eqnarray*}
%
\end{eqnarray*}
If we take $0<\rho<<1,$ we have
\begin{eqnarray*}
%
\end{eqnarray*}
thus, it holds that
\begin{eqnarray*}
%
\end{eqnarray*}
Hence, by applying Lemma \ref{L7} with $\mathbb{E}\int_{0}^{t}\|B^{\varepsilon}_{x}\|^{2p}ds$, we have
\begin{eqnarray*}
%
\end{eqnarray*}
plug this inequality into (\ref{12}), we have
\begin{eqnarray*}
%
\end{eqnarray*}
by using Gronwall inequality, we have
\begin{eqnarray}\label{30}
%
\end{eqnarray}
moreover, we have
\begin{eqnarray}\label{31}
%
\end{eqnarray}
\par
$\bullet$ The estimate of $\mathbb{E}\sup\limits_{0\leq t\leq T}\|A^{\varepsilon}_{x}(t)\|^{2p}.$
\par
Indeed, it follows from (\ref{29}) that
\begin{eqnarray*}
%
\end{eqnarray*}
If we take $0<\rho<<1,$ we have
\begin{eqnarray*}
%
\end{eqnarray*}

It follows from the Young inequality(see Lemma \ref{L9}) that
\begin{eqnarray*}
%
\end{eqnarray*}
In view of the Burkholder-Davis-Gundy inequality as in \cite[P87]{Y1}, it holds that
\begin{eqnarray*}
%
\end{eqnarray*}
where we have use (\ref{30}) in the last inequality.
\par
According to (\ref{26}) and (\ref{31}), it holds that
\begin{eqnarray*}
%
\end{eqnarray*}
It follows from the above estimates, we have
\begin{eqnarray*}
%
\end{eqnarray*}
By using Gronwall inequality, we have
\begin{eqnarray*}
%
\end{eqnarray*}
\end{proof}

\subsection{Well-posedness for the averaged equation (\ref{8})}

\par
By the same method in Theorem \ref{Th2} and Proposition \ref{P1}, we can obtain the following proposition.
\begin{proposition}\label{P3}
If $A_{0}\in H^{1}_{0}(I),$ (\ref{8}) has a unique solution $\bar{A}\in L^{2}(\Omega,C([0,T];H^{1}_{0}(I))).$ Moreover, for any $p>0,$ there exists a constant $C$ such that the solution $\bar{A}$ satisfies
\begin{eqnarray*}
%
\end{eqnarray*}
where $C$ dependent of $p,T,A_{0},B_{0}$ but independent of $p>0.$
\end{proposition}

\section{Proof of Theorem \ref{Th1}}
\subsection{H\"{o}lder continuity of time variable for $A^{\varepsilon}$}
The following proposition is a crucial step.

\begin{proposition}\label{L2}
There exists a constant $C(p,T)$ such that
\begin{eqnarray}\label{11}
%
\end{eqnarray}
for any $t\in [0,T],h>0.$
\end{proposition}
\begin{proof}

\par
Let us write
\begin{eqnarray*}
%
\end{eqnarray*}
here $Id$ denotes the identity operator.

\par
Due to \cite{P1}, there is a $C$ such that for all $x\in H^{1}(I)$,
\begin{eqnarray*}
%
\end{eqnarray*}
and then, according to this inequality and Proposition \ref{P1}, we have
\begin{eqnarray*}
%
\end{eqnarray*}
\par
It follows from Proposition \ref{P1} that
\begin{eqnarray*}
%
\end{eqnarray*}
according to the facts
\begin{eqnarray*}
%
\end{eqnarray*}
it holds that
\begin{eqnarray*}
%
\end{eqnarray*}
according to the facts
\begin{eqnarray*}
%
\end{eqnarray*}
it holds that
\begin{eqnarray*}
%
\end{eqnarray*}
By the same way, we have
\begin{eqnarray*}
%
\end{eqnarray*}
\par
In view of the Burkholder-Davis-Gundy inequality and H\"{o}lder¡¯s inequality, it yields
\begin{eqnarray*}
%
\end{eqnarray*}
\par
By the above estimates, we arrive at (\ref{11}).
\end{proof}

\subsection{Auxiliary process $(\hat{A}^{\varepsilon},\hat{B}^{\varepsilon})$}
\par
Next, we introduce an auxiliary process $(\hat{A}^{\varepsilon},\hat{B}^{\varepsilon})\in H\times H$ by Khasminskii in \cite{K1}.
\par
Fix a positive number $\delta$ and do a partition of time interval $[0,T]$ of size $\delta$. We construct a process $\hat{B}^{\varepsilon}\in L^{2}(I)$ by means of the equations
\begin{eqnarray*}
%
\end{eqnarray*}
for $t\in [k\delta,\min\{(k+1)\delta,T\}),k\geq0.$
\par
Also define
the process $\hat{A}^{\varepsilon}\in L^{2}(I)$ by
\begin{eqnarray*}
%
\end{eqnarray*}
for $t\in [0,T]$, where $s(\delta)=[\frac{s}{\delta}]\delta$ is the nearest breakpoint preceding $s$ and $[\cdot]$ is the integer function.
\par
Thus $(\hat{A}^{\varepsilon},\hat{B}^{\varepsilon})$ satisfies
\begin{equation}\label{9}
%
\end{equation}

\subsection{Some priori estimates of $(\hat{A}^{\varepsilon},\hat{B}^{\varepsilon})$}

Because the proof almost follows the steps in Proposition \ref{P1}, we omit
the proof here.

\begin{proposition}
If $A_{0},B_{0}\in H^{1}_{0}(I),$ for $\varepsilon\in(0,1),$ $(\hat{A}^{\varepsilon},\hat{B}^{\varepsilon})$ is the unique solution to (\ref{9}),  then there exists a constant $C$ such that the solutions $(\hat{A}^{\varepsilon},\hat{B}^{\varepsilon})$ satisfy
\begin{eqnarray*}
%
\end{eqnarray*}
where $C_{1}$ dependent of $T,A_{0},B_{0}$ but independent of $\varepsilon\in (0,1).$
\par
Moreover, for any $p>0,$ there exists a constant $C_{2}$ such that
\begin{eqnarray*}
%
\end{array}
\end{eqnarray*}
where $C_{2}$ dependent of $p,T,A_{0},B_{0}$ but independent of $\varepsilon\in (0,1),$ $p>0.$
\end{proposition}

\subsection{The errors of $A^{\varepsilon}-\hat{A}^{\varepsilon}$ and $B^{\varepsilon}-\hat{B}^{\varepsilon}$}
\par
We will establish the convergence of the auxiliary process $\hat{B}^{\varepsilon}$ to the fast solution process $B^{\varepsilon}$ and $\hat{A}^{\varepsilon}$ to the slow solution process $A^{\varepsilon}$, respectively.
\begin{lemma}\label{L3}
There exists a constant $C$ such that
\begin{eqnarray*}

\end{eqnarray*}
where $C$ is only dependent of $p,T,A_{0},B_{0}.$
\end{lemma}
\begin{proof}
$\bullet$ We prove the first inequality.
\par
Indeed,
we have
\begin{equation*}
%
\end{equation*}
it is easy to see that $B^{\varepsilon}(t)-\hat{B}^{\varepsilon}(t)$ satisfies the following SPDE
\begin{equation}\label{4}
%
\end{eqnarray*}
By taking mathematical expectation from both sides of above equation, we have
\begin{eqnarray*}
%
\end{eqnarray*}
thus, we have
\begin{eqnarray*}
%
\end{eqnarray*}
It follows from Lemma \ref{L1}, we have
\begin{eqnarray*}
%
\end{eqnarray*}
thus, it holds that
\begin{eqnarray*}
%
\end{eqnarray*}
It follows from the Young inequality(see Lemma \ref{L9}) that
\begin{eqnarray*}
%
\end{eqnarray*}
due to Lemma \ref{L2}, it holds that
\begin{eqnarray*}
%
\end{eqnarray*}
Hence, by applying Lemma \ref{L7} with $\mathbb{E}\|(B^{\varepsilon}-\hat{B}^{\varepsilon})(t)\| ^{2p}$, we have
\begin{eqnarray*}
%
\end{eqnarray*}
\par
$\bullet$ We prove the second inequality.
\par
Indeed, we have
\begin{equation*}
%
\end{equation*}
it is easy to see that $A^{\varepsilon}(t)-\hat{A}^{\varepsilon}(t)$ satisfies the following SPDE
\begin{equation}\label{28}
%
\end{eqnarray*}
It follows from the Young inequality(see Lemma \ref{L9}) that
\begin{eqnarray*}
%
\end{eqnarray*}
\par
It follows from Lemma \ref{L8}, Proposition \ref{L2} and Proposition \ref{P1} that
\begin{eqnarray*}
%
\end{eqnarray*}
by the same method, we have
\begin{eqnarray*}
%
\end{eqnarray*}
\end{proof}
\subsection{The errors of $\hat{A}^{\varepsilon}-\bar{A}$}
\par
Next we prove the strong convergence of the auxiliary process $\hat{A}^{\varepsilon}$ to the averaging solution process $\bar{A}$.
\begin{lemma}\label{L4}
There exists a constant $C(T,p)$ such that
\begin{eqnarray*}
%
\end{equation*}
\par In mild sense, we introduce the following decomposition
\begin{eqnarray*}
%
\end{eqnarray*}

\par
$\bullet$ For $J_{1},$ we can rewrite $J_{1}$ as
\begin{eqnarray*}
%
\end{eqnarray*}
\par
$\star$ For $J_{11},$ by using the H\"{o}lder inequality, we have
\begin{eqnarray*}
%
\end{eqnarray*}
\par
It follows from Proposition \ref{P1} and Proposition \ref{L2} that
\begin{eqnarray*}
%
\end{eqnarray*}

\par
$\star$ For $J_{12},$ by using the H\"{o}lder inequality and the same method in dealing with $J_{11},$ we have
\begin{eqnarray*}
%
\end{eqnarray*}
\par
It follows from Proposition \ref{P1} and Proposition \ref{L2} that
\begin{eqnarray*}
%
\end{eqnarray*}
\par
$\star$ For $J_{13},$ by using the H\"{o}lder inequality and the same method in dealing with $J_{11},$ we have
\begin{eqnarray*}
%
\end{eqnarray*}
\par
In order to deal with the above estimate, we will use the skill of stopping times, this is inspired from \cite{D2}.
\par
We define the stopping time
\begin{eqnarray*}
%
\end{eqnarray*}

\par
$\bullet$ For $J_{2},$ we can rewrite $J_{2}$ as
\begin{eqnarray*}
%
\end{eqnarray*}
\par
$\star$ For $J_{21},$ by using the H\"{o}lder inequality, we have
\begin{eqnarray*}
%
\end{eqnarray*}
\par
It follows from Proposition \ref{P1} and Proposition \ref{L2} that
\begin{eqnarray*}
%
\end{eqnarray*}

\par
$\star$ For $J_{22},$ by using the H\"{o}lder inequality and the same method in dealing with $J_{11},$ we have
\begin{eqnarray*}
%
\end{eqnarray*}
\par
$\star$ For $J_{23},$ by using the same method in dealing with $J_{13},$ we have
\begin{eqnarray*}
%
\end{eqnarray*}
\par
It holds that
\begin{eqnarray*}
%
\end{eqnarray*}
Noting the fact that
when $0\leq t\leq \tau_{n}^{\varepsilon},$
it holds that
\begin{eqnarray*}
%
\end{eqnarray*}
thus, we have
\begin{eqnarray*}
%
\end{eqnarray*}

\par
$\bullet$ For $J_{3},$ we can rewrite $J_{3}$ as
\begin{eqnarray*}
%
\end{eqnarray*}
where $m_{t}=[\frac{t}{\delta}]$.
\par
$\star$ For $J_{32},$ due to Lemma \ref{L2}, it concludes that
\begin{eqnarray*}
%
\end{eqnarray*}
\par
$\star$ For $J_{33},$ according to Proposition \ref{P1} and the global Lipschitz property of $f$ and $\bar{f},$ we have
\begin{eqnarray*}
%
\end{eqnarray*}

\par
$\star$ For $J_{34},$ due to Lemma \ref{L2}, it concludes that
\begin{eqnarray*}
%
\end{eqnarray*}
\par
$\star$ For $J_{35},$ it concludes that
\begin{eqnarray*}
%
\end{eqnarray*}
\par
$\star$ For $J_{31},$ by a time shift transformation, we have
for any fixed $p$ and $t\in[0,\delta)$ that
\begin{eqnarray*}
%
\end{eqnarray*}
where $W_{2}^{\ast}(t)$ is the shift version of $W_{2}(t)$ and hence they have the same distribution.
\par
Let $\bar{W}(t)$ be a Wiener process defined on the same stochastic basis and independent
of $W_{1}(t)$ and $W_{2}(t)$. Construct a process $B^{A^{\varepsilon}(p\delta),B^{\varepsilon}(p\delta)}(t)\in L^{2}(I)$ by means of
\begin{eqnarray*}
%
\end{eqnarray*}
here $\bar{\bar{W}}(t)$ is the scaled version of $\bar{W}(t).$
\par
By comparing the above two equations, we see that
\begin{eqnarray}\label{15}
%
\end{eqnarray}
where $\sim$ denotes a coincidence in distribution sense. Thus, we have
\begin{eqnarray*}
%
\end{eqnarray*}
It follows from (\ref{15}) and the property of semigroup $\{S(t)\}_{t\geq0}$ that
\begin{eqnarray*}
%
\end{eqnarray*}
\par
 In view of the above inequality, Lemma \ref{L6} and the method in \cite{F1,F2,F3,F4}, there holds
\begin{eqnarray*}
%
\end{eqnarray*}
Thus if choose $\delta=\delta(\varepsilon)$ such that $\frac{\delta}{\varepsilon}$ sufficiently large, we have

\begin{eqnarray*}
%
\end{eqnarray*}
\par
On the other hand, it holds that
\begin{eqnarray*}
%
\end{eqnarray*}
With the help of the above estimates, we have
\begin{eqnarray*}
%
\end{eqnarray*}
By using the Gronwall inequality, we have
\begin{eqnarray*}
%
\end{eqnarray*}
this implies that
\begin{eqnarray*}
%
\end{eqnarray*}
\par
On the other hand, due to Proposition \ref{P1}, we have
\begin{eqnarray*}
%
\end{eqnarray*}
\par
Hence, we have
\begin{eqnarray*}
%
\end{eqnarray*}
if we take $n=\sqrt[4p]{-\frac{1}{8C}\ln \varepsilon},\delta=\varepsilon^{\frac{1}{2}},$ we obtain
\begin{eqnarray*}
%
\end{eqnarray*}

\end{proof}

\subsection{Proof of Theorem \ref{Th1}}
By taking $\delta=\varepsilon^{\frac{1}{2}},$ we have
\begin{eqnarray*}
%
\end{eqnarray*}
thus, we have
\begin{eqnarray*}
%
\end{eqnarray*}
\par
If $0<p\leq\frac{5}{4},$ for any $\kappa>0,$ it holds that
\begin{eqnarray*}
%
\end{array}
\end{eqnarray*}

\par
This completes the proof of Theorem \ref{Th1}.

\par

\baselineskip 9pt \renewcommand{\baselinestretch}{1.08}
\par
\par


{\small
\begin{thebibliography}{99}

\bibitem{A1}Atai J., Malomed B.A., Stability and interactions of solitons in two-component active systems[J], Phys Rev E 1996;54:4371.


\bibitem{B0} Bogoliubov N.N., Mitropolsky Y.A., Asymptotic Methods in the Theory of Non-linear Oscillations[M], Gordon \&
Breach Science Publishers, New York, 1961.


\bibitem{B1} Barton-Smith M., Invariant measure for the stochastic Ginzburg Landau equation[J], Nonlinear Differential Equations and Applications NoDEA, 2004, 11(1): 29-52.

\bibitem{B2} Br\'{e}hier C.E., Strong and weak orders in averaging for SPDEs[J], Stochastic Process. Appl. 122 (2012) 2553-2593.

\bibitem{B3} Barton-Smith M. Global solution for a stochastic Ginzburg-Landau equation with multiplicative noise[J]. Stochastic analysis and applications, 2004, 22(1): 1-18.

\bibitem{B4} Bao J, Yin G, Yuan C. Two-time-scale stochastic partial differential equations driven by $\alpha$-stable noises: Averaging principles[J]. Bernoulli, 2017, 23(1): 645-669.


\bibitem{C1} Cerrai S. and Freidlin M. I., Averaging principle for a class of stochastic reaction diffusion
equations[J], Probab. Th. Relat. Fields 144 (2009) 137-177.

\bibitem{C2} Cerrai S., A Khasminkii type averaging principle for stochastic reaction-diffusion
equations[J], Ann. Appl. Probab. 19 (2009) 899-948.

\bibitem{C3} Cerrai S., Averaging Principle for Systems of Reaction-Diffusion Equations with Polynomial Nonlinearities Perturbed by Multiplicative Noise[J], SIAM Journal on Mathematical Analysis, 2011, 43(6): 2482-2518.

\bibitem{C4} Chow P L. Stochastic partial differential equations[M]. CRC Press, 2014.


\bibitem{D1} Da Prato G., Zabczyk J., Stochastic equations in infinite dimensions[M], Cambridge university press, 2014.

\bibitem{D2} Dong Z, Sun X, Xiao H, et al. Averaging principle for one dimensional stochastic Burgers equation[J]. arXiv preprint arXiv:1701.05920, 2017.


\bibitem{F1} Fu H., Wan L., Liu J., Strong convergence in averaging principle for stochastic hyperbolic-parabolic equations with two time-scales[J], Stochastic Process. Appl., 2015, 125(8): 3255-3279.

\bibitem{F2} Fu H., Duan J., An averaging principle for two-scale stochastic partial differential equations[J], Stochastics and Dynamics, 2011, 11(02n03): 353-367.

\bibitem{F3} Fu H., Wan L., Wang Y., Liu J., Strong convergence rate in averaging principle for stochastic FitzHugh-Nagumo system with two time-scales[J], Journal of Mathematical Analysis and Applications, 2014, 416(2): 609-628.

\bibitem{F4} Fu H., Liu J., Strong convergence in stochastic averaging principle for two time-scales stochastic partial differential equations[J], Journal of Mathematical Analysis and Applications, 2011, 384(1): 70-86.

\bibitem{F7} Fu H, Wan L, Liu J, et al. Weak order in averaging principle for stochastic wave equations with a fast oscillation[J]. arXiv preprint arXiv:1701.07984, 2017.


\bibitem{F5} Freidlin M.I., Wentzell A.D., Random Perturbation of Dynamical Systems[M], second ed., Springer-Verlag, New York, 1998.

\bibitem{F6} Freidlin M.I., Wentzell A.D., Long-time behavior of weakly coupled oscillators[J], J. Stat. Phys. 123 (2006) 1311-1337.

\bibitem{G1} Golec J., Ladde G., Averaging principle and systems of singularly perturbed stochastic differential equations[J], J. Math. Phys. 31 (1990) 1116-1123.

\bibitem{G2} Givon D., Kevrekidis I.G., Kupferman R., Strong convergence of projective integration schemes for singular perturbed stochastic differential systems[J], Commun. Math. Sci. 4 (2006) 707-729.

\bibitem{G3} Givon D., Strong convergence rate for two-time-scale jump-diffusion stochastic differential systems[J], SIAM J. Multiscale Model. Simul. 6 (2) (2007) 577-594.

\bibitem{G4} Golec J., Stochastic averaging principle for systems with pathwise uniqueness[J], Stoch. Anal. Appl. 13 (3) (1995) 307-322.



\bibitem{L1} Lions, J.L., Magenes, E., Non-Homogeneous Boundary Value Problems and Applications[M], vol.I, Grundlehren Math. Wiss., Band 181, Springer-Verlag, NewYork-Heidelberg, translated fromthe French by P.Kenneth(1972).

\bibitem{L2} Li X.M., An averaging principle for a completely integrable stochastic Hamiltonian system[J], Nonlinearity 21 (2008) 803-822.

\bibitem{L3} Lisei H, Keller D. A stochastic nonlinear Schr\"{o}dinger problem in variational formulation[J]. Nonlinear Differential Equations and Applications NoDEA, 2016, 23(2): 1-27.

\bibitem{M1} Malomed B.A., Winful H.G., Stable solitons in two-component active systems[J], Phys Rev E 1996;53:5365.


\bibitem{M2}Marti-Panameno E., Gomez-Pavon L.C., Luis-Ramos A., Mendez-Otero M.M., Castillo M.D.I. Self-mode-locking action in a dualcore
ring fiber laser[J], Opt Commun 2001;194:409.

\bibitem{K1}Khasminskii R.Z., On the principle of averaging the It\^{o} stochastic differential equations[J], Kibernetika. 4 (1968) 260-279 (in Russian).

\bibitem{P1} Pazy A., Semigroups of Linear Operators and Applications to Partial Differential Equations[M], Springer, Berlin, 1985.

\bibitem{P2} Pr\'{e}v\^{o}t C, R\"{o}ckner M. A concise course on stochastic partial differential equations[M]. Berlin: Springer, 2007.

\bibitem{P3} Pei B, Xu Y, Wu J L. Two-time-scales hyperbolic-parabolic equations driven by Poisson random measures: Existence, uniqueness and averaging principles[J]. Journal of Mathematical Analysis and Applications, 2017, 447(1): 243-268.

\bibitem{S1} Sakaguchi H., Phase Dynamics of the Coupled Complex Ginzburg-Landau Equations[J], Progress of Theoretical Physics, 1995, 93(3): 491-502.
\bibitem{S2} Simon H A, Ando A. Aggregation of variables in dynamic systems[J]. Econometrica: journal of the Econometric Society, 1961: 111-138.

\bibitem{V1} Veretennikov A.Y., On the averaging principle for systems of stochastic differential equations[J], Math. USSR-Sb. 69 (1991) 271-284.

\bibitem{V2} Veretennikov A.Y., On large deviations in the averaging principle for SDEs with full dependence[J], Ann. Probab. 27 (1999) 284-296.

\bibitem{W1}Wang W. and Roberts A.J., Average and deviation for slow-fast stochastic partial differential equations[J], J. Differential Equations. 253 (2012) 1265-1286.

\bibitem{W2} Winful H.G., Walton D.T., Passive mode locking through nonlinear coupling in dual-core fiber laser[J], Opt Lett 1992;17:1688.

\bibitem{W3}Walton D.T., Winful H.G., Passive mode locking with an active nonlinear directional coupler: positive group velocity dispersion[J],
Opt Lett 1993;18:720.

\bibitem{X1} Xu J, Miao Y, Liu J. Strong averaging principle for two-time-scale non-autonomous stochastic FitzHugh-Nagumo system with jumps[J]. Journal of Mathematical Physics, 2016, 57(9): 092704.

\bibitem{X2} Xu J. $L^{p}$-strong convergence of the averaging principle for slow¨Cfast SPDEs with jumps[J]. Journal of Mathematical Analysis and Applications, 2017, 445(1): 342-373.


\bibitem{Y1} Yang D., Hou Z., Large deviations for the stochastic derivative Ginzburg-Landau equation with multiplicative noise[J], Physica D: Nonlinear Phenomena, 2008, 237(1): 82-91.



\end{thebibliography}}
\end{document}